\documentclass[english,11pt,reqno]{smfart}

\usepackage{color}
\usepackage{amsthm}
\usepackage{amsmath}
\usepackage{amsfonts}
\usepackage{amstext}
\usepackage[latin1]{inputenc}
\usepackage{amscd}
\usepackage{latexsym}
\usepackage{bm}
\usepackage{amssymb}
\usepackage[all]{xy}
\usepackage{euscript}
\usepackage{a4wide}
\usepackage{amsmath,amssymb,graphicx}
\usepackage{amssymb}
\usepackage{amsmath}
\usepackage{mathrsfs,mathtools}

\newtheorem{theorem}{Theorem}

\newtheorem{property}[theorem]{Property}
\newtheorem{lemma}[theorem]{Lemma}
\newtheorem{corollary}[theorem]{Corollary}
\newtheorem{example}[theorem]{Example}
\newtheorem{definition}[theorem]{Definition}
\newtheorem{remark}[theorem]{Remark}

\def\di{\displaystyle}

\def\eps{\varepsilon}

\def\({\left(} 
\def\){\right)}

\newcommand{\N}{\mathbb{N}}
\newcommand{\R}{\mathbb{R}}
\newcommand{\RN}{(\mathbb{R}^d)^{N+1}}

\newcommand{\LL}{\mathcal{L}}
\newcommand{\CC}{\mathscr{C}}
\newcommand{\Q}{\boldsymbol{Q}}
\newcommand{\Qa}{\boldsymbol{Q}^{\boldsymbol{U},\alpha}}
\newcommand{\U}{\boldsymbol{U}}
\newcommand{\PP}{\boldsymbol{P}}
\newcommand{\PPa}{\boldsymbol{P}^{\boldsymbol{U},\alpha}}
\newcommand{\T}{\boldsymbol{T}}

\newcommand{\G}{\boldsymbol{G}}
\newcommand{\cDM}{{}_{{\rm c}} D^\alpha_-}
\newcommand{\cDP}{{}_{{\rm c}} D^\alpha_+}
\newcommand{\DM}{D^\alpha_-}
\newcommand{\DP}{D^\alpha_+}
\newcommand{\cDDM}{{}_{{\rm c}} \Delta^\alpha_-}
\newcommand{\cDDP}{{}_{{\rm c}} \Delta^\alpha_+}
\newcommand{\DDM}{\Delta^\alpha_-}
\newcommand{\DDP}{\Delta^\alpha_+}

\newcommand{\fonction}[5]{\begin{array}[t]{lrcl}#1 :&#2 &\longrightarrow &#3\\&#4& \longmapsto &#5 \end{array}}

\newcommand{\fonctionsansdef}[3]{\begin{array}[t]{lrcl}#1 :&#2 &\longrightarrow &#3 \end{array}}

\newcommand{\abs}[1]{\left\vert #1 \right\vert} 

\begin{document}
\title[]{\textbf{Variational integrator for fractional Pontryagin's systems. Existence of a discrete fractional Noether's theorem.}}
\author{Lo\"ic Bourdin}
\address{Laboratoire de Math\'ematiques et de leurs Applications - Pau (LMAP). UMR CNRS 5142. Universit\'e de Pau et des Pays de l'Adour.}
\email{bourdin.l@etud.univ-pau.fr}
\maketitle

\begin{abstract}
Fractional Pontryagin's systems emerge in the study of a class of fractional optimal control problems, see \cite{agra2,agra3,bour7,torr,torr2,jeli} and references therein, but they are not resolvable in most cases. In this paper, we suggest a numerical approach for these fractional systems. Precisely, we construct a variational integrator allowing to preserve at the discrete level their intrinsic variational structure. The variational integrator obtained is then called \textit{shifted discrete fractional Pontryagin's system}. 

In \cite{bour7}, we have provided a solved fractional example in a certain sense. It allows us to test in this paper the convergence of the variational integrator constructed. Finally, we also provide a discrete fractional Noether's theorem giving the existence of an explicit computable discrete constant of motion for shifted discrete fractional Pontryagin's systems admitting a discrete symmetry.
\end{abstract}



\section*{Introduction}
The fractional calculus, \textit{i.e.} the mathematical field dealing with the generalization of the derivative to any real order, plays an increasing role in many varied domains as economy \cite{comt} or probability \cite{levy,stan}. Fractional derivatives also appear in many fields of Physics (see \cite{hilf3}):  wave mechanic \cite{alme}, viscoelasticity \cite{bagl}, thermodynamics \cite{hilf2}, fluid mechanic in heterogeneous media \cite{hilf,neel,neel2}, etc. Recently, a subtopic of the fractional calculus gains importance: it concerns the variational principles on functionals involving fractional derivatives. This leads to the statement of fractional Euler-Lagrange equations, see \cite{agra,bale2,riew}. 

A direct consequence is the emergence of works concerning a particular class of fractional optimal control problems, see \cite{agra2,agra3,bour7,torr,torr2,jeli} and references therein. Using a Lagrange multiplier technique or not, authors obtain with a calculus of variations a necessary condition for the existence of an optimal control. This condition is commonly given as the existence of a solution of a system of fractional differential equations called \textit{fractional Pontryagin's system}. 

Hence, the explicit computation of a potential optimal control, from the above necessary condition, needs the resolution of a fractional Pontryagin's system which is a main drawback. Indeed, solving a fractional differential equation is in general very difficult. Moreover, a fractional Pontryagin's system involves left and right fractional derivatives which is an additional obstruction. \\

In this paper, we then develop a \textit{numerical approach}. Let us remind that there exist many works concerning the statement of discrete operators approaching the fractional derivatives (see \cite{dubo,galu,oldh}) and then concerning numerical schemes for fractional differential equations (see \cite{diet,liu,meer,meer2}). In particular, one can find studies concerning the discretization of fractional Euler-Lagrange equations \cite{agra4,bour} and fractional Pontryagin's systems \cite{agra2,agra3,bale,deft,jeli}. 

Nevertheless, a fractional Pontryagin's system admits an intrinsic variational structure: its solutions correspond to the critical points of a cost functional. Moreover, this variational structure induces strong constraints on the qualitative behaviour of the solutions and it seems then important to preserve it at the discrete level. A \textit{variational integrator} is a numerical scheme preserving the variational structure of a system at the discrete level. We refer to Section \ref{section2} for more details concerning the construction of a variational integrator and let us remind that the variational integrators are well-developed in \cite{lubi,mars} for classical Euler-Lagrange equations and in \cite{bour} for fractional ones. In this chapter, we construct a variational integrator for fractional Pontryagin's systems and it is called shifted discrete fractional Pontryagin's system. \\

In \cite{bour7}, we have suggested a deviously way in order to get informations on the solutions of a not solvable fractional Pontryagin's system. Indeed, we have stated a fractional Noether's theorem giving an explicit constant of motion for fractional Pontryagin's systems admitting a symmetry. We refer to \cite{bour7} for more details and we remind that this result is based on a preliminary result proved by Torres and Frederico in \cite{torr,torr2}. In this paper, following the strategy of the continuous case, we introduce the notion of a discrete symmetry for shifted discrete fractional Pontryagin's systems and we finally provide a discrete fractional Noether's theorem giving an explicit computable constant of motion. \\

The paper is organized as follows. Section \ref{section1} is devoted to a reminder on the fractional calculus and on the emergence of fractional Pontryagin's systems in the study of a class of fractional optimal control problems. In Section \ref{section2}, after a reminder concerning discrete fractional derivatives, we focus on the construction of a variational integrator for fractional Pontryagin's systems. We make some numerical tests in Section \ref{section3}. Especially, let us remind that a fractional example is solved in \cite{bour7} in a certain sense. Consequently, we can test the convergence of the variational integrator both in the classical and strict fractional cases. Finally, Section \ref{section4} is devoted to the statement of a discrete fractional Noether's theorem. Technical proofs of Lemmas are provided in Appendix \ref{appA}.

\parindent 0pt

\section{Reminder about fractional Pontryagin's system}\label{section1}
In this section, we first make a reminder about fractional calculus in Section \ref{section11}. Then, in Section \ref{section12}, we remind how fractional Pontryagin's systems emerge from the study of a class of fractional optimal control problems. Let us introduce the following notations available in the whole paper. Let $a < b$ be two reals, let $d$, $m \in \N^*$ denote two dimensions and let $\Vert \cdot \Vert$ be the euclidean norm of $\R^d$ and $\R^m$.

\subsection{Fractional operators of Riemann-Liouville and Caputo}\label{section11}
The fractional calculus concerns the extension of the usual notion of derivative from non-negative integer orders to any real order. Since 1695, numerous notions of fractional derivatives emerge over the year, see \cite{kilb,podl,samk}. In this paper, we only use the notions of fractional integrals and derivatives in the sense of Riemann-Liouville (1847) and Caputo (1967) whose definitions are recalled in this section. We refer to \cite{kilb,podl,samk} for more details. \\

Let $g \in \CC^0 ([a,b],\R^d)$ and $\alpha > 0$. The left (resp. right) fractional integral in the sense of Riemann-Liouville with inferior limit $a$ (resp. superior limit $b$) of order $ \alpha $ of $g$ is defined by:
\begin{equation}
\forall t \in ]a,b], \; I^{\alpha}_- g (t) := \dfrac{1}{\Gamma (\alpha)} \di \int_a^t (t-y)^{\alpha -1} g(y) \; dy
\end{equation}
respectively:
\begin{equation}
\forall t \in [a,b[, \; I^{\alpha}_+ g (t) := \dfrac{1}{\Gamma (\alpha)} \di \int_t^b (y-t)^{\alpha -1} g(y) \; dy,
\end{equation}
where $\Gamma$ denotes the Euler's Gamma function. For $\alpha =0$, let $I^0_- g = I^0_+ g = g$. \\

Now, let us consider $0 < \alpha \leq 1$. The left (resp. right) fractional derivative in the sense of Riemann-Liouville with inferior limit $a$ (resp. superior limit $b$) of order $\alpha $ of $g$ is then given by:
\begin{equation}
\forall t \in ]a,b], \; \DM g(t) := \dfrac{d}{dt} \big( I^{1-\alpha}_- g \big) (t) \quad \Big( \text{resp.} \quad \forall t \in [a,b[, \; \DP g(t) := -\dfrac{d}{dt} \big( I^{1-\alpha}_+ g \big) (t) \Big),
\end{equation}
provided that the right side terms are defined. \\

In the Riemann-Liouville sense, the strict fractional derivative of a constant is not zero. Caputo then suggests the following definition. For $0 < \alpha \leq 1$, the left (resp. right) fractional derivative in the sense of Caputo with inferior limit $a$ (resp. superior limit $b$) of order $\alpha $ of $g$ is given by:
\begin{equation}\label{eq11-1}
\forall t \in ]a,b], \; \cDM g(t) := \DM \big( g - g(a) \big) (t) \quad \Big( \text{resp.} \quad \forall t \in [a,b[, \; \cDP g(t) := \DP \big( g - g(b) \big) (t) \Big),
\end{equation}
provided that the right side terms are defined. Let us note that if $g(a)=0$ (resp. $g(b)=0$), then $\cDM g = \DM g$ (resp. $\cDP g =  \DP g$).\\

In the \textit{classical case} $\alpha = 1$, the fractional derivatives of Riemann-Liouville and Caputo both coincide with the classical derivative. Precisely, modulo a $(-1)$ term in the right case, we have $D^1_- = {}_{\text{c}}D^1_- = - D^1_+ = - {}_{\text{c}}D^1_+ = d/dt$. 

\subsection{Reminder about a class of fractional optimal control problems}\label{section12}
From now and for all the rest of the paper, we consider $0 < \alpha \leq 1$ and $A \in \R^d$. Let us denote $[\alpha]$ the floor of $\alpha$. \\

In this section, let us remind the following definitions concerning the class of fractional optimal control problems studied in \cite{bour7}:
\begin{itemize}
\item The elements denoted $u \in \CC^0 ([a,b],\R^m)$ are called \textit{controls};
\item Let $f$ be a $\CC^2$ function of the form:
\begin{equation}
\fonction{f}{\R^d \times \R^m \times [a,b]}{\R^d}{(x,v,t)}{f(x,v,t).}
\end{equation}
It is commonly called the \textit{constraint function}. We assume that $f$ satisfies the following Lipschitz type condition. There exists $M \geq 0$ such that:
\begin{equation}\tag{$f_x$ lip}\label{condf}
\forall (x_1,x_2,v,t) \in (\R^d)^2 \times \R^m \times [a,b], \; \Vert f(x_1,v,t) - f(x_2,v,t) \Vert \leq M \Vert x_1 - x_2 \Vert ;
\end{equation}
\item For any control $u$, let $q^{u,\alpha} \in \CC^{[\alpha]} ([a,b],\R^d)$ denote the unique global solution of the following fractional Cauchy problem:
\begin{equation}\tag{CP${}^\alpha_q$}\label{eqcpq}
 \left\lbrace \begin{array}{l}
         \cDM q =f(q,u,t)\\
	     q(a) = A.
        \end{array}
\right.
\end{equation}
$q^{u,\alpha}$ is commonly called the \textit{state variable} associated to $u$. Its existence and its uniqueness are provided in \cite{bour7} from Condition \eqref{condf};
\item Finally, the fractional optimal control problem studied in \cite{bour7} is the problem of optimization of the following \textit{cost functional}:
\begin{equation}
\fonction{\LL^\alpha}{\CC^0 ([a,b],\R^m )}{\R}{u}{\di \int_{a}^{b} L ( q^{u,\alpha},u,t ) \; dt ,} 
\end{equation}
where $L$ is a \textit{Lagrangian}, \textit{i.e.} a $\CC^2$ application of the form:
\begin{equation}
\fonction{L}{\R^d \times \R^m \times [a,b]}{\R}{(x,v,t)}{L(x,v,t).}
\end{equation}
\end{itemize}

A control optimizing $\LL^\alpha$ is called \textit{optimal control}. A necessary condition for a control $u$ to be optimal is to be a \textit{critical point} of $\LL^\alpha$, \textit{i.e.} to satisfy:
\begin{equation}
\forall \bar{u} \in  \CC^0 ([a,b],\R^m ), \;  D\LL^\alpha(u)(\bar{u}) := \lim\limits_{\eps \rightarrow 0} \dfrac{\LL^\alpha(u+\eps \bar{u})-\LL^\alpha(u)}{\eps} = 0. 
\end{equation}
In \cite{bour7}, we then focused on the characterization of the critical points of $\LL^\alpha$. Firstly, we proved with an usual calculus of variations the following Lemma \ref{lem1} giving explicitly the value of the G\^ateaux derivative of $\LL^\alpha$:
\begin{lemma}\label{lem1}
Let $u$, $\bar{u} \in \CC^0 ([a,b],\R^m )$. Then, the following equality holds:
\begin{equation}
D\LL^\alpha(u)(\bar{u}) = \di \int_a^b \dfrac{\partial L}{\partial x} (q^{u,\alpha},u,t) \cdot \bar{q} + \dfrac{\partial L}{\partial v} (q^{u,\alpha},u,t) \cdot \bar{u} \; dt,
\end{equation}
where $\bar{q} \in \CC^{[\alpha]} ( [a,b], \R^d )$ is the unique global solution of the following linearised Cauchy problem:
\begin{equation}\label{eqlcpq}\tag{LCP${}^\alpha_{\bar{q}}$}
 \left\lbrace \begin{array}{l}
 		\cDM \bar{q} = \dfrac{\partial f}{\partial x} (q^{u,\alpha},u,t) \times \bar{q} + \dfrac{\partial f}{\partial v} (q^{u,\alpha},u,t) \times \bar{u} \\[10pt]
 		\bar{q}(a) = 0.
        \end{array}
\right. 
\end{equation}
\end{lemma}

This last result not leading to a characterization of the critical points of $\LL^\alpha$, we then introduced the following elements stemming from the Lagrange multiplier technique:
\begin{itemize}
\item Let $H$ be the following application
\begin{equation}
\fonction{H}{\R^d \times \R^m \times \R^d \times [a,b]}{\R}{(x,v,w,t)}{L(x,v,t)+w \cdot f(x,v,t).}
\end{equation}
$H$ is commonly called the \textit{Hamiltonian} associated to the Lagrangian $L$ and the constraint function $f$;
\item For any control $u$, let $p^{u,\alpha} \in \CC^{[\alpha]} ( [a,b], \R^d )$ denote the unique global solution of the following fractional Cauchy problem:
\begin{equation}\label{eqcpp}\tag{CP${}^\alpha_p$}
 \left\lbrace \begin{array}{l}
         \cDP p = \dfrac{\partial H}{\partial x}(q^{u,\alpha},u,p,t) = \dfrac{\partial L}{\partial x}(q^{u,\alpha},u,t) + \left( \dfrac{\partial f}{\partial x}(q^{u,\alpha},u,t)\right)^T \times p \\[10pt]
	     p(b) = 0.
        \end{array}
\right.
\end{equation}
$p^{u,\alpha}$ is commonly called the \textit{adjoint variable} associated to $u$. Its existence and its uniqueness are also provided in \cite{bour7}. Let us note that $\cDP p^{u,\alpha} = \DP p^{u,\alpha}$ since $p^{u,\alpha} (b) = 0$.
\end{itemize}
Consequently, for any control $u$, the couple $(q^{u,\alpha},p^{u,\alpha})$ is solution of the following \textit{fractional Hamiltonian system}:
\begin{equation}\tag{HS${}^\alpha$}\label{eqhs}
 \left\lbrace \begin{array}{l}
 		 \cDM q = \dfrac{\partial H}{\partial w} (q,u,p,t) \\[10pt]
 		\DP p = \dfrac{\partial H}{\partial x} (q,u,p,t).
        \end{array}
\right.
\end{equation}
Finally, the introduction of these last elements allowed us to prove the following theorem:
\begin{theorem}\label{thmfinal}
Let $u \in \CC^0 ([a,b],\R^m)$. Then, $u$ is a critical point of $\LL^{\alpha}$ if and only if $(q^{u,\alpha},u,p^{u,\alpha})$ is solution of the following \textit{fractional stationary equation}:
\begin{equation}\tag{SE${}^\alpha$}\label{eqse}
\dfrac{\partial H}{\partial v} (q,u,p,t) = 0.
\end{equation}
\end{theorem}

From Theorem \ref{thmfinal}, we retrieved in \cite{bour7} the following result leading to the fractional Pontryagin's system:
\begin{corollary}\label{corfinal}
$\LL^\alpha$ has a critical point in $\CC^0 ([a,b],\R^m)$ if and only if there exists $(q,u,p) \in \CC^{[\alpha]} ([a,b],\R^d) \times \CC^0 ([a,b],\R^m) \times \CC^{[\alpha]} ([a,b],\R^d)$ solution of the following \textit{fractional Pontryagin's system}:
\begin{equation}\tag{PS${}^\alpha $}\label{eqps}
 \left\lbrace \begin{array}{l}
 		\cDM q = \dfrac{\partial H}{\partial w} (q,u,p,t) \\[10pt]
 		\DP p = \dfrac{\partial H}{\partial x} (q,u,p,t) \\[10pt]
	    \dfrac{\partial H}{\partial v} (q,u,p,t) = 0 \\[10pt]
	    \big( q(a),p(b) \big) = ( A,0 ).
        \end{array}
\right.
\end{equation}
In the affirmative case, $u$ is a critical point of $\LL^\alpha$ and we have $(q,p) = (q^{u,\alpha},p^{u,\alpha})$.
\end{corollary}

Let us note that the fractional Pontryagin's system \eqref{eqps} is made up of the fractional Hamiltonian system \eqref{eqhs}, the fractional stationary equation \eqref{eqse} and initial and final conditions. \\

In practice, see Examples in \cite{bour7}, we use more Corollary \ref{corfinal} than Theorem \ref{thmfinal}. Let us remind that Corollary \ref{corfinal} was already provided in \cite{agra2,agra3,torr,torr2,jeli} and references therein without Condition \eqref{condf}. However, this result is proved, in each of these papers, using a Lagrange multiplier technique requiring the introduction of an augmented functional. In \cite{bour7}, Condition \eqref{condf} allowed us to give a complete proof of this result using only classical mathematical tools adapted to the fractional case: \textit{calculus of variations}, \textit{Gronwall's Lemma}, \textit{Cauchy-Lipschitz Theorem} and \textit{stability under perturbations of differential equations}. We refer to \cite{bour7} for more details and for a discussion on the subject. \\

As we have seen in this section, fractional Pontryagin's systems emerge from the study of a class of fractional optimal control problems. They have a variational structure in the sense that they are obtained with a calculus of variations on functionals and there resolutions give explicitly the critical points of these functionals. Our aim in this paper is to provide them numerical schemes preserving this strong characteristic at the discrete level. \\

Moreover, let us make the following important remark: since a fractional Pontryagin's system emerges from a fractional optimal control problem, the main unknown is then the control $u$. Consequently, the convergence of the numerical scheme constructed in Section \ref{section2} is going to be considered only with respect to $u$.

\section{Variational integrator for fractional Pontryagin's systems}\label{section2}
In general, fractional differential equations are very difficult to solve. One can find some solved examples in \cite{kilb,podl,samk} using Mittag-Leffler functions, Fourier and Laplace transforms. Additionally, fractional Pontryagin's systems, as fractional Euler-Lagrange equations provided in \cite{agra}, present an asymmetry in the sense that left and right fractional derivatives are involved. It is an additional drawback in order to solve explicitly the most of fractional Pontryagin's systems. In this section, we then develop a numerical approach treating them. \\

Nevertheless, as we have seen in Section \ref{section12}, a fractional Pontryagin's system admits an intrinsic variational structure: its solutions correspond to the critical points of a functional. In this paper, we want to construct a numerical scheme for fractional Pontryagin's systems preserving at the discrete level this strong property. \\ 

A \textit{variational integrator} is a numerical scheme preserving the variational structure of a system at the discrete level. Precisely, let us consider a differential system coming from a variational principle (\textit{i.e.} \textit{its solutions correspond to the critical points of a functional}). Then, a variational integrator is the numerical scheme constructed as follows:
\begin{itemize}
\item firstly, one have to define a discrete version of the functional;
\item secondly one have to form a discrete variational principle on it.
\end{itemize}
Hence, a numerical scheme is obtained and it is called variational integrator. It preserves the variational structure at the discrete level in the sense that \textit{its discrete solutions correspond to the discrete critical points of the discrete functional}. Let us remind that variational integrators are well-developed for classical Euler-Lagrange equations in \cite{lubi,mars} and let us remind that we have developed a variational integrator for fractional Euler-Lagrange equations in \cite{bour}. In this section, we are going to construct a variational integrator for fractional Pontryagin's systems. \\

Let us introduce the following notations available in the whole paper. Let $N \in \N^*$, $h=(b-a)/N$ denote the step size of discretization and $\T = (t_k)_{k=0,\ldots,N} = (a+kh)_{k=0,\ldots,N}$ be the classical partition of the interval $[a,b]$. Let us assume that $N$ is sufficiently large in order to satisfy the following condition:
\begin{equation}\label{condh}\tag{cond $h$}
2h^\alpha M < 1,
\end{equation}
where $M$ is the Lipschitz coefficient of the constraint function $f$, see Condition \eqref{condf}.

\subsection{Reminder about discrete fractional derivatives of Gr\"unwald-Letnikov}\label{section21}
For the sequel, we need the introduction of discrete operators approximating the fractional derivatives of Riemann-Liouville and Caputo. As in \cite{deft,dubo}, let us define $\DDM$ and $\DDP$ the following discrete analogous of $\DM$ and $\DP$ respectively:
\begin{equation}
\fonction{\DDM}{( \R^d ) ^{N+1}}{( \R^d ) ^{N}}{\G}{\left( \dfrac{1}{h^{\alpha}} \di \sum_{r=0}^{k} \alpha_r G_{k-r} \right) _{k=1,\ldots,N}}
\end{equation}
and
\begin{equation}
\fonction{\DDP}{( \R^d ) ^{N+1}}{( \R^d ) ^{N}}{\G}{\left( \dfrac{1}{h^{\alpha}} \displaystyle \sum_{r=0}^{N-k} \alpha_r G_{k+r} \right) _{k=0,\ldots,N-1},}
\end{equation}
where the elements $(\alpha_r)_{r \in \N}$ are defined by $\alpha_0 := 1$ and 
\begin{equation}\label{eqalphar}
\forall r \in \N^*, \; \alpha_r := \dfrac{(-\alpha)(1-\alpha)\ldots(r-1-\alpha)}{r!}.
\end{equation}
These discrete fractional operators are approximations of the continuous ones. Indeed, passing to the limit $h \to 0$, these discrete operators correspond to the definition of the fractional derivatives of Gr\"unwald-Letnikov (1867) coinciding with the Riemann-Liouville's ones. We refer to \cite{podl} for more details. \\

Finally, according to Equation \eqref{eq11-1}, we define $\cDDM$ and $\cDDP$ the following discrete analogous of $\cDM$ and $\cDP$ respectively:
\begin{equation}
\fonction{\cDDM}{( \R^d ) ^{N+1}}{( \R^d ) ^{N}}{\G}{\Big( \big( \DDM (\G - G_0) \big)_k \Big) _{k=1,\ldots,N}}
\end{equation}
and
\begin{equation}
\fonction{\cDDP}{( \R^d ) ^{N+1}}{( \R^d ) ^{N}}{\G}{\Big( \big( \DDP (\G - G_N) \big)_k \Big) _{k=0,\ldots,N-1}.}
\end{equation}

Let us note that we preserve some continuous properties at the discrete level. In particular, $G_0 = 0$ (resp. $G_N = 0$) implies $\cDDM \G = \DDM \G$ (resp. $\cDDP \G = \DDP \G$). Additionally, in the classical case $ \alpha = 1$, these discrete fractional derivatives coincide with the usual backward and forward Euler's approximations of $d/dt$ with a $(-1)$ term in the right case:
\begin{equation}
\forall k=1,\ldots,N , \; (\Delta^1_- \G )_k= ({}_{{\rm c}} \Delta^1_- \G)_k = \dfrac{G_k - G_{k-1}}{h}
\end{equation}
and
\begin{equation}
\forall k=0,\ldots,N-1 , \; (\Delta^1_+ \G )_k= ({}_{{\rm c}} \Delta^1_+ \G)_k = \dfrac{G_k - G_{k+1}}{h}.
\end{equation}

\subsection{Results concerning the discrete fractional derivatives}\label{section22}
In this section, we prove two important properties preserved from the continuous level to the discrete one. For the sequel, we first need the introduction of the following \textit{shift operators}:
\begin{equation}
\fonction{\sigma}{(\R^n)^{N+1}}{(\R^n)^{N}}{\G}{\big( G_{k+1} \big)_{k=0,\ldots,N-1}} \quad \text{and} \quad \fonction{\sigma^{-1}}{(\R^n)^{N+1}}{(\R^n)^{N}}{\G}{ \big( G_{k-1} \big)_{k=1,\ldots,N},} 
\end{equation}
where the integer $n$ is $d$ or $m$. \\

The first property is the following: considering the quadrature formula of Gauss as approximation of the integral, we can prove the following discrete fractional integration by parts:

\begin{property}[Discrete fractional integration by parts]\label{lemdfibp}
Let $\G^1$, $\G^2 \in \RN$ satisfying $G^1_0 = G^2_N = 0$, then we have:
\begin{equation}\tag{DFIBP}\label{eqdfibp}
h \di \sum_{k=1}^N (\cDDM \G^1)_k \cdot \sigma^{-1} (\G^2)_k = h \di \sum_{k=0}^{N-1} \sigma (\G^1)_k \cdot (\cDDP \G^2)_k.
\end{equation}
\end{property}

\begin{proof} 
Since $G^1_0 =  G^2_N = 0$, we have $\cDDM \G^1 = \DDM \G^1 $ and $\cDDP \G^2 = \DDP \G^2 $. Then, we have:
\begin{equation}
h \di \sum_{k=1}^N (\DDM \G^1)_k \cdot \sigma^{-1} (\G^2)_k = h \di \sum_{k=0}^{N-1} (\DDM \G^1)_{k+1} \cdot G^2_k = h^{1-\alpha} \di \sum_{k=0}^{N-1} \sum_{r=0}^{k+1} \alpha_r G^1_{k+1-r} \cdot G^2_k.
\end{equation}
Finally, since $G^1_0 =  G^2_N = 0$, the following equalities hold:
\begin{equation}
\begin{array}{rccl}
& h \di \sum_{k=1}^N (\DDM \G^1)_k \cdot \sigma^{-1} (\G^2)_k & = & h^{1-\alpha} \di \sum_{k=0}^{N-1} \sum_{r=0}^{k} \alpha_r G^1_{k+1-r} \cdot G^2_k \\
= & h^{1-\alpha} \di \sum_{r=0}^{N-1} \sum_{k=r}^{N-1} \alpha_r G^1_{k+1-r} \cdot G^2_k & = & h^{1-\alpha} \di \sum_{r=0}^{N-1} \sum_{k=0}^{N-r-1} \alpha_r G^1_{k+1} \cdot G^2_{k+r} \\
 = & h^{1-\alpha} \di \sum_{k=0}^{N-1} G^1_{k+1} \cdot \left( \sum_{r=0}^{N-k-1} \alpha_r G^2_{k+r} \right) & = & h^{1-\alpha} \di \sum_{k=0}^{N-1} G^1_{k+1} \cdot \left( \sum_{r=0}^{N-k} \alpha_r G^2_{k+r} \right) ,
\end{array}
\end{equation}
which concludes the proof.
\end{proof}

This last result is very useful for discrete calculus of variations involving discrete fractional derivatives, see proof of Theorem \ref{thmfinald}. Secondly let us prove the following discrete version of the fractional Cauchy-Lipschitz Theorem proved in \cite{bour7}:

\begin{theorem}[Discrete fractional Cauchy-Lipschitz theorem]\label{thmdfcl}
Let $F \in \CC^0(\R^d \times [a,b], \R^d)$ satisfying the following Lipschitz type condition:
\begin{equation}\label{eq22}
\exists K \in \R, \; \forall (x_1,x_2,t) \in (\R^d)^2 \times [a,b], \; \Vert F(x_1,t)-F(x_2,t) \Vert \leq K \Vert x_1 - x_2 \Vert,
\end{equation}
with $h^\alpha K < 1$. Then, the following discrete fractional Cauchy problem:
\begin{equation}\label{eq21}
 \left\lbrace \begin{array}{l}
 		\cDDM \Q = F(\Q,\T) \\
 		Q_0 = A
        \end{array}
\right.
\end{equation}
has an unique solution $\Q \in \RN$.
\end{theorem}

\begin{proof}
We are going to construct by induction the solution $\Q$ of \eqref{eq21}. Our method uses the classical fix point theorem concerning the contraction mappings. Indeed, let us choose $Q_0 = A$. Then, for any $k=1,\ldots,N$, $Q_k$ has to satisfy:
\begin{equation}
Q_k = h^\alpha F(Q_k,t_k) + Q_0 - \di \sum_{r=1}^{k-1} \alpha_r (Q_{k-r}-Q_0).
\end{equation}
However, for any $k=1,\ldots,N$, the application $ h^\alpha F(\cdot,t_k) + Q_0 - \sum_{r=1}^{k-1} \alpha_r (Q_{k-r}-Q_0) $ is a contraction and consequently admits an unique fix point. Hence, we first construct $Q_1$, then $Q_2$, etc. By induction, we construct a solution $\Q$ of \eqref{eq21} and such a construction assures its uniqueness.
\end{proof}

\subsection{First step of construction}\label{section23}
As said in introduction of this section, in order to complete the first step of construction of a variational integrator, we have to provide a discrete version of $\LL^\alpha$. In this way, let us give the following definition:
\begin{itemize}
\item The elements $\U \in (\R^m)^{N+1}$ are called the \textit{discrete controls};
\item For any discrete control $\U$, let $\Qa \in \RN$ denote the unique solution of the following discrete Cauchy problem:
\begin{equation}\tag{CP${}^\alpha_{\Q}$}\label{eqdcpq}
 \left\lbrace \begin{array}{l}
         \cDDM \Q =f(\Q,\U,\T)\\
	     Q_0 = A \in \R^d .
        \end{array}
\right.
\end{equation}
$\Qa$ is called the \textit{discrete state variable} associated to $\U$. Its existence and its uniqueness are provided by Theorem \ref{thmdfcl} and Conditions \eqref{condf} and \eqref{condh};
\item Finally, we define the following \textit{discrete cost functional}:
\begin{equation}
\fonction{\LL^\alpha_h}{(\R^m)^{N+1}}{\R}{\U}{h \di \sum_{k=1}^N L(Q^{\U,\alpha}_k,U_k,t_k).}
\end{equation}
\end{itemize}
Hence, we have provided a discrete version $\LL^\alpha_h$ to the cost functional $\LL^\alpha$. Now, the second step of the construction of the variational integrator is to characterize the \textit{discrete critical points} of the discrete cost functional $\LL^\alpha_h$ with the help of a discrete calculus of variations. \\

Let us make the following remark: such a characterization implies to be a necessary condition for the existence of an optimizer of the discrete cost functional $\LL^\alpha_h$. In fact, in this section, we have defined an actual discrete fractional optimal control problem.

\subsection{Second step of construction}\label{section24}
The second step of construction of a variational integrator consists in forming a discrete variational principle on $\LL^\alpha_h$. Precisely, we focus on the characterization of its \textit{discrete critical points}, \textit{i.e.} the elements $\U \in (\R^m)^{N+1}$ satisfying:
\begin{equation}
\forall \bar{\U} \in (\R^m)^{N+1}, \; D\LL^\alpha_h (\U)(\bar{\U}) := \lim\limits_{\eps \to 0} \dfrac{\LL^\alpha_h (\U+\eps \bar{\U})-\LL^\alpha_h (\U)}{\eps} = 0.
\end{equation}
With a discrete calculus of variations, we obtain the following discrete version of Lemma \ref{lem1} giving explicitly the value of the G\^ateaux derivative of $\LL^\alpha_h$.

\begin{lemma}\label{lem2}
Let $\U$, $\bar{\U} \in (\R^m)^{N+1}$. Then, the following equality holds:
\begin{equation}
D\LL^\alpha_h (\U)(\bar{\U}) = h \di \sum_{k=1}^N  \left[ \dfrac{\partial L}{\partial x} (Q^{\U,\alpha}_k, U_k,t_k) \cdot \bar{Q}_k + \dfrac{\partial L}{\partial v} (Q^{\U,\alpha}_k, U_k,t_k) \cdot \bar{U}_k \right],
\end{equation}
where $\bar{\Q} \in (\R^d )^{N+1}$ is the unique solution of the following linearised discrete fractional Cauchy problem:
\begin{equation}\label{eqldcpq}\tag{LCP${}^\alpha_{\bar{\Q}}$}
 \left\lbrace \begin{array}{l}
 		\cDDM \bar{\Q} = \dfrac{\partial f}{\partial x} (\Qa,\U,\T) \times \bar{\Q} + \dfrac{\partial f}{\partial v} (\Qa,\U,\T) \times \bar{\U} \\[10pt]
 		\bar{Q}_0 = 0.
        \end{array}
\right. 
\end{equation}
\end{lemma}

\begin{proof}
See Appendix \ref{appAc}.
\end{proof}

This last result does not lead to a characterization of the critical points of $\LL^\alpha_h$ yet. As in the continuous case, we then introduce the notion of discrete adjoint variable: for any discrete control $\U$, let $\PP^{\U,\alpha} \in \RN$ denote the unique solution of the following shifted discrete Cauchy problem:
\begin{equation}\tag{$\sigma$CP${}^\alpha_{\PP}$}
 \left\lbrace \begin{array}{rcl}
        \cDDP \PP & = & \dfrac{\partial H}{\partial x} \big( \sigma (\Q^{\U,\alpha}),\sigma (\U),\PP,\sigma (\T) \big) \\[10pt]
        & = & \dfrac{\partial L}{\partial x} \big( \sigma (\Q^{\U,\alpha}),\sigma (\U),\sigma (\T) \big) + \left(  \dfrac{\partial f}{\partial x} \big( \sigma (\Q^{\U,\alpha}),\sigma (\U),\sigma (\T) \big) \right)^T \times \PP \\[10pt]
	    P_N & = & 0.
        \end{array}
\right.
\end{equation}
$\PP^{\U,\alpha}$ is called the \textit{discrete adjoint variable} associated to $\U$. Its existence and its uniqueness are provided by the analogous of Theorem \ref{thmdfcl} for right discrete fractional derivative and by Conditions \eqref{condf} and \eqref{condh}. Let us note that, since $\PP^{\U,\alpha}_N = 0$, we can write $\cDDP \PP^{\U,\alpha} = \DDP \PP^{\U,\alpha}$. \\

The presence of shift operators in the definition of the discrete adjoint variable is the consequence of the change of sums in the discrete fractional integration by parts \eqref{eqdfibp} (see Property \ref{lemdfibp}). We refer to the proof of Theorem \ref{thmfinald} for more details. We also refer to Remark \ref{rem0} for a discussion about the presence of the shift operators. \\

Finally, let us note that for any discrete control $\U$, the couple $(\Qa,\PPa)$ is solution of the following \textit{shifted discrete fractional Hamiltonian system}:
\begin{equation}\tag{$\sigma$HS${}^\alpha_h$}\label{eqhamsystd}
 \left\lbrace \begin{array}{l}
   \cDDM \Q = \dfrac{\partial H}{\partial w} \big(  \Q,\U,\sigma^{-1}(\PP),\T \big) \\[10pt]
        \DDP \PP = \dfrac{\partial H}{\partial x} \big( \sigma (\Q),\sigma (\U),\PP,\sigma (\T) \big).
	   \end{array}
\right.
\end{equation}
Finally, the introduction of this last discrete element allows us to prove the following theorem:
\begin{theorem}\label{thmfinald}
Let $\U \in (\R^m)^{N+1}$. Then, $\U$ is a discrete critical point of $\LL^{\alpha}_h$ if and only if $(\Qa,\U,\PPa)$ is solution of the following \textit{shifted discrete fractional stationary equation}:
\begin{equation}\tag{$\sigma$SE${}^\alpha_h$}\label{eqstatfd}
\dfrac{\partial H}{\partial v} \big( \Q,\U,\sigma^{-1} (\PP),\T \big) = 0.
\end{equation}
\end{theorem}

\begin{proof}
Let $\U$, $\bar{\U} \in (\R^m)^{N+1}$. From Lemma \ref{lem2}, we have:
\begin{multline}
h^{-1} D\LL^\alpha_h (\U)(\bar{\U}) = \di \sum_{k=1}^N  \left[ \dfrac{\partial L}{\partial x} (Q^{\U,\alpha}_k, U_k,t_k) + \left(  \dfrac{\partial f}{\partial x} (Q^{\U,\alpha}_k, U_k,t_k) \right)^T \times \sigma^{-1} (\PPa)_k \right] \cdot \bar{Q}_k \\ - \di \sum_{k=1}^N  \left( \left(  \dfrac{\partial f}{\partial x} (Q^{\U,\alpha}_k, U_k,t_k) \right)^T \times \sigma^{-1} (\PPa)_k \right)\cdot \bar{Q}_k + \di \sum_{k=1}^N \dfrac{\partial L}{\partial v} (Q^{\U,\alpha}_k, U_k,t_k) \cdot \bar{U}_k .
\end{multline}
Then:
\begin{multline}
h^{-1} D\LL^\alpha_h (\U)(\bar{\U}) = \di \sum_{k=0}^{N-1} (\cDDP \PPa)_k \cdot \sigma (\bar{\Q})_k - \di \sum_{k=1}^N  \left( \dfrac{\partial f}{\partial x} (Q^{\U,\alpha}_k, U_k,t_k) \times \bar{Q}_k \right) \cdot \sigma^{-1} (\PPa)_k \\ + \di \sum_{k=1}^N \dfrac{\partial L}{\partial v} (Q^{\U,\alpha}_k, U_k,t_k) \cdot \bar{U}_k .
\end{multline}
From the discrete fractional integration by parts \eqref{eqdfibp} (see Property \ref{lemdfibp}), we obtain:
\begin{multline}
h^{-1} D\LL^\alpha_h (\U)(\bar{\U}) = \di \sum_{k=1}^{N} \left( (\cDDM \bar{\Q})_k - \dfrac{\partial f}{\partial x} (Q^{\U,\alpha}_k, U_k,t_k) \times \bar{Q}_k \right) \cdot \sigma^{-1} (\PPa)_k \\ + \di \sum_{k=1}^N \dfrac{\partial L}{\partial v} (Q^{\U,\alpha}_k, U_k,t_k) \cdot \bar{U}_k .
\end{multline}
Since $\bar{\Q}$ is solution of \eqref{eqldcpq}, we have:
\begin{eqnarray*}
h^{-1} D\LL^\alpha_h (\U)(\bar{\U}) & = & \di \sum_{k=1}^{N} \left( \dfrac{\partial f}{\partial v} (Q^{\U,\alpha}_k, U_k,t_k) \times \bar{U}_k \right) \cdot \sigma^{-1} (\PPa)_k + \di \sum_{k=1}^N \dfrac{\partial L}{\partial v} (Q^{\U,\alpha}_k, U_k,t_k) \cdot \bar{U}_k \\
& = & \di \sum_{k=1}^{N} \left( \left( \dfrac{\partial f}{\partial v} (Q^{\U,\alpha}_k, U_k,t_k) \right)^T \times \sigma^{-1} (\PPa)_k + \dfrac{\partial L}{\partial v} (Q^{\U,\alpha}_k, U_k,t_k) \right) \cdot \bar{U}_k.
\end{eqnarray*}
Finally:
\begin{equation}
D\LL^\alpha_h(\U)(\bar{\U}) = h \di \sum_{k=1}^{N} \dfrac{\partial H}{\partial v}  \big(Q^{\U,\alpha}_k,U_k,\sigma^{-1}(\PPa)_k,t_k \big) \cdot \bar{U}_k.
\end{equation}
The proof is completed.
\end{proof}

Finally, from Theorem \ref{thmfinald}, we obtain the following result leading to the variational integrator constructed:
\begin{corollary}\label{corfinald}
$\LL^\alpha_h$ has a discrete critical point if and only if there exists $(\Q,\U,\PP) \in (\R^d)^{N+1} \times (\R^m)^{N+1} \times (\R^d)^{N+1}$ solution of the following \textit{shifted discrete fractional Pontryagin's system}:
\begin{equation}\tag{$\sigma$PS${}^\alpha_h$}\label{eqpontsystfd}
 \left\lbrace \begin{array}{l}
 		\cDDM \Q = \dfrac{\partial H}{\partial w} \big(\Q,\U,\sigma^{-1}(\PP),\T\big) \\[10pt]
 		\DDP \PP = \dfrac{\partial H}{\partial x} \big(\sigma(\Q),\sigma(\U),\PP,\sigma(\T)\big) \\[10pt]
	    \dfrac{\partial H}{\partial v} \big( \Q,\U,\sigma^{-1}(\PP),\T \big) = 0 \\[10pt]
	    ( Q_0,P_N ) = ( A,0 ).
        \end{array}
\right.
\end{equation}
In this case, $\U$ is a discrete critical point of $\LL^\alpha_h$ and we have $(\Q,\PP) = (\Qa,\PPa)$.
\end{corollary}
Let us note that \eqref{eqpontsystfd} is made up of the shifted discrete Hamiltonian system \eqref{eqhamsystd}, the shifted stationary equation \eqref{eqstatfd} and initial and final conditions. \\

Hence, we have constructed the variational integrator \eqref{eqpontsystfd} for the fractional Pontryagin's system \eqref{eqps}. It is then a numerical scheme for \eqref{eqps} preserving its variational structure in the sense that the discrete solutions $\U$ obtained correspond to the discrete critical points of the discrete version $\LL^\alpha_h$ of $\LL^\alpha$.

\begin{remark}\label{rem0}
Let us note that the variational integrator \eqref{eqpontsystfd} does not correspond with a direct discretization of \eqref{eqps} as it is done in \cite{deft}. There is an emergence of shift operators caused by the conservation at the discrete level of the variational structure. However, it is proved that the use of shifted numerical schemes allows to obtain more stability for some fractional differential equations, see \cite{meer2,meer}. 
\end{remark}

\begin{remark}\label{rem1}
Let us remind the following remark: since a fractional Pontryagin's system emerges from a fractional optimal control problem, the main unknown is then the control $u$. Consequently, the convergence of the variational integrator \eqref{eqpontsystfd} is going to be considered only with respect to $u$. Let us note that the value of $U_0$ does not take place in the variational integrator \eqref{eqpontsystfd}: it is a free value. Nevertheless, this is totally coherent with the fact that this value does not take place neither in the definition of $\LL^\alpha_h$. Hence, in the following examples in Section \ref{section3}, the error between an exact solution $u$ of \eqref{eqps} and a numerical solution $\U$ obtained with \eqref{eqpontsystfd} is going to be evaluated on $\Vert u(t_k) -U_k \Vert$ for $k \in \{ 1,\ldots,N \}$ only.
\end{remark}

\subsection{Link with the discrete fractional Euler-Lagrange equation}\label{section25}
Let us take the constraint function $f(x,v,t) = v$ satisfying \eqref{condf}. In this case, applying Corollary \ref{corfinald}, we know that there exists a critical point of $\LL^\alpha_h$ if and only if there exists a solution $(\Q,\U,\PP) \in (\R^d)^{N+1} \times (\R^m)^{N+1} \times (\R^d)^{N+1}$ of the shifted discrete fractional Pontryagin's system \eqref{eqpontsystfd} here given by:
\begin{equation}
 \left\lbrace \begin{array}{l}
 		\cDDM \Q = \U \\[10pt]
 		\DDP \PP = \dfrac{\partial L}{\partial x} \big(\sigma(\Q),\sigma(\U),\sigma(\T)\big) \\[10pt]
	    \dfrac{\partial L}{\partial v} ( \Q,\U,\T ) + \sigma^{-1} (\PP) = 0 \\[10pt]
	    ( Q_0,P_N ) = ( A,0 ).
        \end{array}
\right.
\end{equation}
In the affirmative case, it implies that $\Q$ is a discrete solution of the following \textit{discrete fractional Euler-Lagrange equation}:
\begin{equation}\tag{EL${}^\alpha_h$}
 \dfrac{\partial L}{\partial x} ( \Q,\cDDM \Q,\T ) + \DDP \left(  \dfrac{\partial L}{\partial v} ( \Q,\cDDM \Q,\T ) \right) = 0.
\end{equation}
Finally, according to our works in \cite{bour}, we then obtain that $\Q$ is a critical point of the following \textit{discrete fractional Lagrangian functional}:
\begin{equation}
\Q \longrightarrow h \di \sum_{k=1}^N L \big(Q_k,(\cDDM \Q)_k,t_k \big).
\end{equation}
We refer to \cite{bour} for more details concerning discrete fractional Euler-Lagrange equations.

\section{Numerical tests}\label{section3}
In the following numerical tests, according to Remark \ref{rem1}, we are going to give graphic representations only of discrete solutions $\U$ and the study of the convergence of the variational integrator \eqref{eqpontsystfd} is only going to be evaluated on the convergence of the discrete control to the continuous one.

\subsection{The linear-quadratic example}\label{section31}
Linear-quadratic examples are often studied in the literature because they are used for tracking problems. The aim of these problems is to determine a control allowing to approach as much as possible reference trajectories, \cite[Part 1.4, p.49]{trel}. In this section, we study such an example, \cite[Part 4.4.3, example 3, p.53]{evan}. More generally, a quadratic Lagrangian is often natural (for example in order to minimize distances) and even if the constraint functions are frequently non linear, we are often leaded to study linearised versions. \\

Let us choose $d = m = A = 1 $ and $[a,b] = [0,1]$. Then, let us take the following quadratic Lagrangian and linear constraint function:
\begin{equation}\label{eq31-1}
\fonction{L}{\R^2 \times [0,1]}{\R}{(x,v,t)}{(x^2+v^2)/2} \quad \text{and} \quad \fonction{f}{\R^2 \times [0,1]}{\R}{(x,v,t)}{x+v.}
\end{equation}
Let us give the graphic representations of the numerical solutions $\U$ given by \eqref{eqpontsystfd} for $N=500$ and for $\alpha=1$, $3/4$, $1/2$, $1/4$: 
\begin{equation*}
\begin{array}{c}
\includegraphics[width=0.6\textwidth]{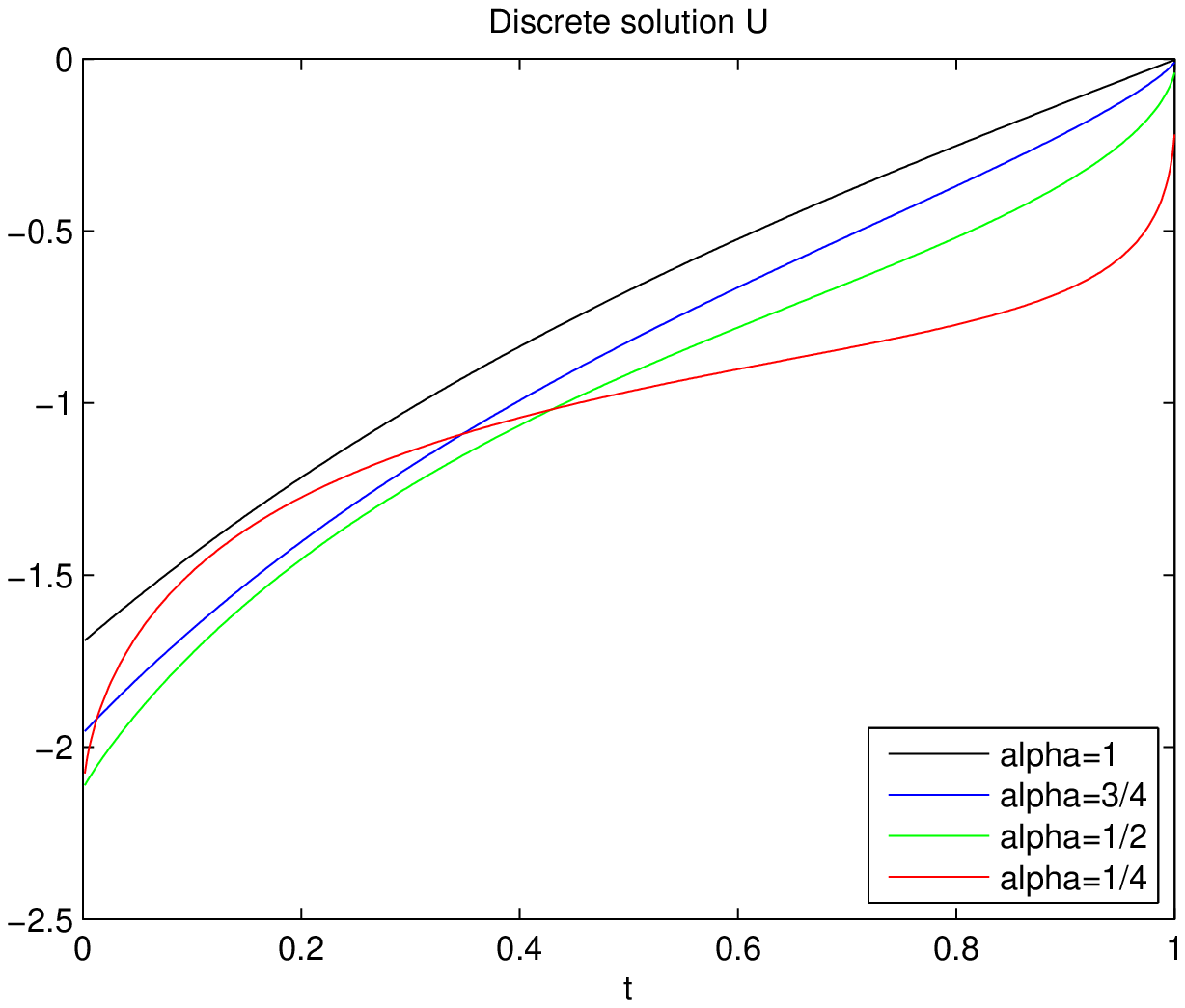}
\end{array}
\end{equation*}
We have seen in \cite{bour7} that the fractional Pontryagin's system \eqref{eqps} is explicitly solved only in the classical case $\alpha =1$ and we obtained the following unique critical point of $\mathcal{L}^1$:
\begin{equation}
\forall t \in [0,1], \; u(t) = \dfrac{\cosh(\sqrt{2})}{R} \sinh (\sqrt{2}t) - \dfrac{\sinh(\sqrt{2})}{R} \cosh (\sqrt{2}t),
\end{equation}
where $ R = \sqrt{2} \cosh (\sqrt{2}) - \sinh (\sqrt{2}) $. Hence, we can only test the convergence of the variational integrator \eqref{eqpontsystfd} for $\alpha =1$. We give the following graphic representing the logarithm of the error $\max \big( \vert u(t_k)-U_k \vert, k=1,\ldots,N \big)$ versus the logarithm of $h$ and the identity function for comparison:
\begin{equation*}
\begin{array}{c}
\includegraphics[width=0.6\textwidth]{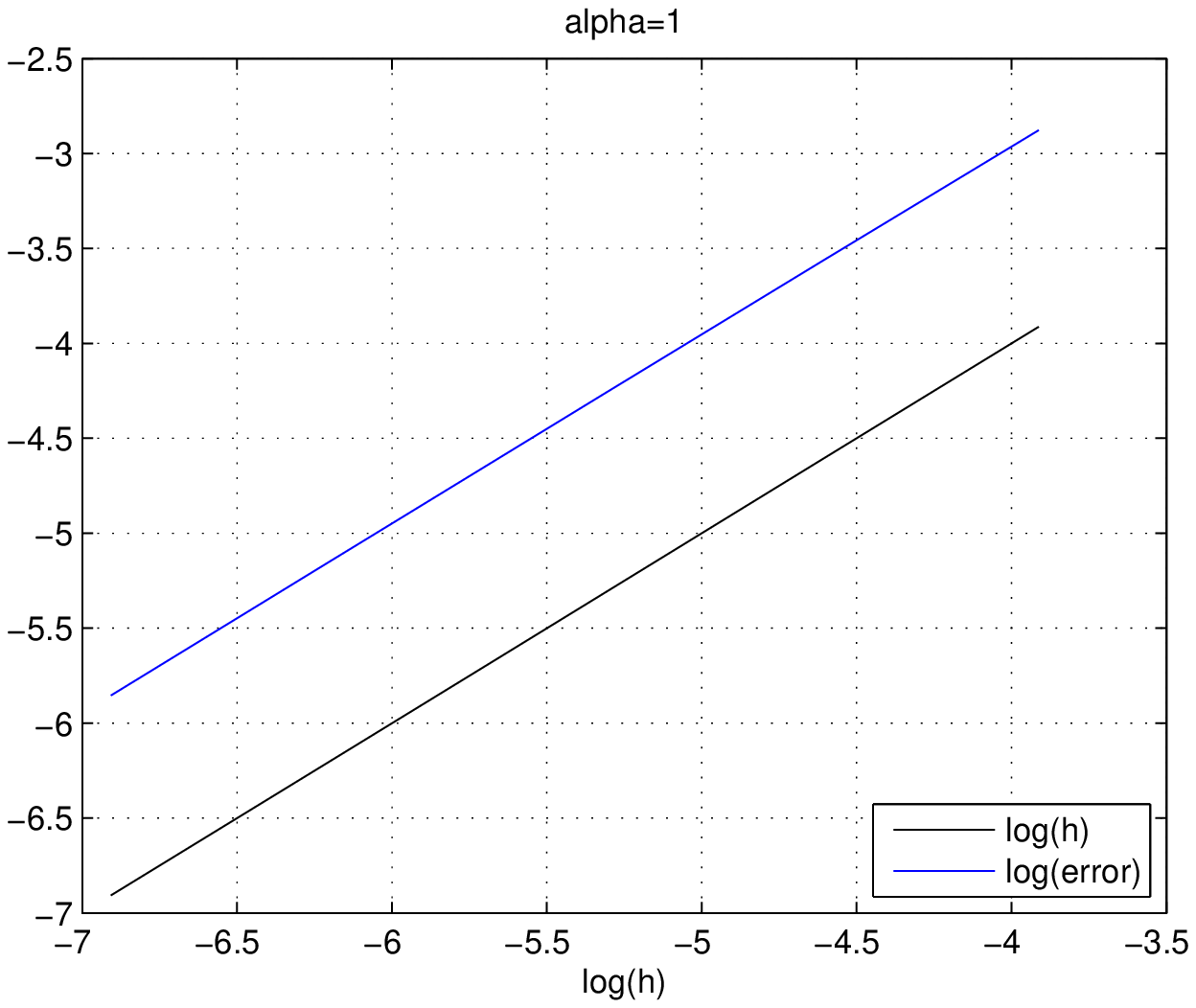}
\end{array}
\end{equation*}
In this example with $\alpha = 1$, the convergence seems then obtained with order $1$. Nevertheless, we do not know the exact solution of \eqref{eqps} in the strict fractional case $0<\alpha <1$. Consequently, we can not study the behaviour of the error in this case. 

\subsection{A solved fractional example}\label{section32} 
In this section, we are going to compute \eqref{eqpontsystfd} in the framework of an example solved in the strict fractional case in the sense that we know explicitly the unique critical point $u$ of $\LL^\alpha_h$ for any $0< \alpha \leq 1$, see \cite{bour7}. Consequently, for this example, we can test the convergence of the variational integrator \eqref{eqpontsystfd} for any $0 < \alpha \leq 1$. \\

Then, let us choose $d=m=A=1$ and $[a,b]=[0,1]$. Then, let us take the following Lagrangian and linear constraint function:
\begin{equation}\label{eq31-1}
\fonction{L}{\R^2 \times [0,1]}{\R}{(x,v,t)}{(1-t)x+(v^2/2)} \quad \text{and} \quad \fonction{f}{\R^2 \times [0,1]}{\R}{(x,v,t)}{x+v.}
\end{equation}

Let us give the graphic representations of the numerical solutions $\U$ given by \eqref{eqpontsystfd} for $N=500$ and for $\alpha=1$, $3/4$, $1/2$, $1/4$: 
\begin{equation*}
\begin{array}{c}
\includegraphics[width=0.6\textwidth]{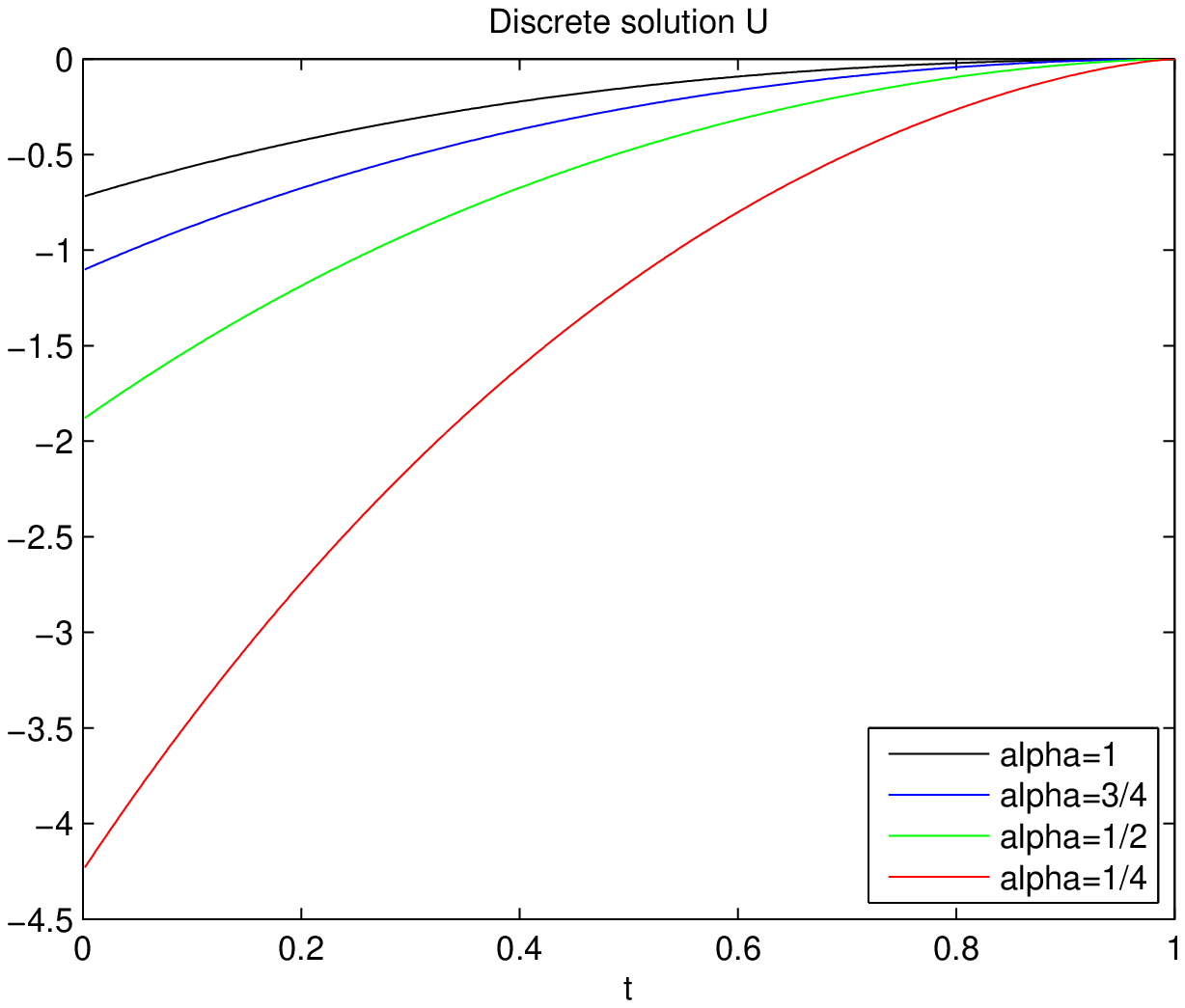}
\end{array}
\end{equation*}
As we have seen in \cite{bour7}, the fractional Pontryagin's system \eqref{eqps} is explicitly solved for any $0 < \alpha \leq 1$ and we obtained the following unique critical point of $\LL^\alpha$:
\begin{equation}
\forall t \in [0,1], \; u(t) = -(1-t)^{\alpha +1} {\rm E}_{\alpha,\alpha+2} \big( (1-t)^\alpha \big),
\end{equation}
where ${\rm E}_{\alpha,\alpha+2}$ is the Mittag-Leffler function with parameter $(\alpha,\alpha+2)$. Let us the convergence of the variational integrator \eqref{eqpontsystfd} for any $0<\alpha \leq 1$. We give the following graphics representing the logarithm of the error $\max \big( \vert u(t_k)-U_k \vert, k=1,\ldots,N \big)$ versus the logarithm of $h$ and the identity function for comparison for $\alpha=1, \; 3/4, \; 1/2, \; 1/4$:
\begin{equation*}
\begin{array}{cc}
\includegraphics[width=0.5\textwidth]{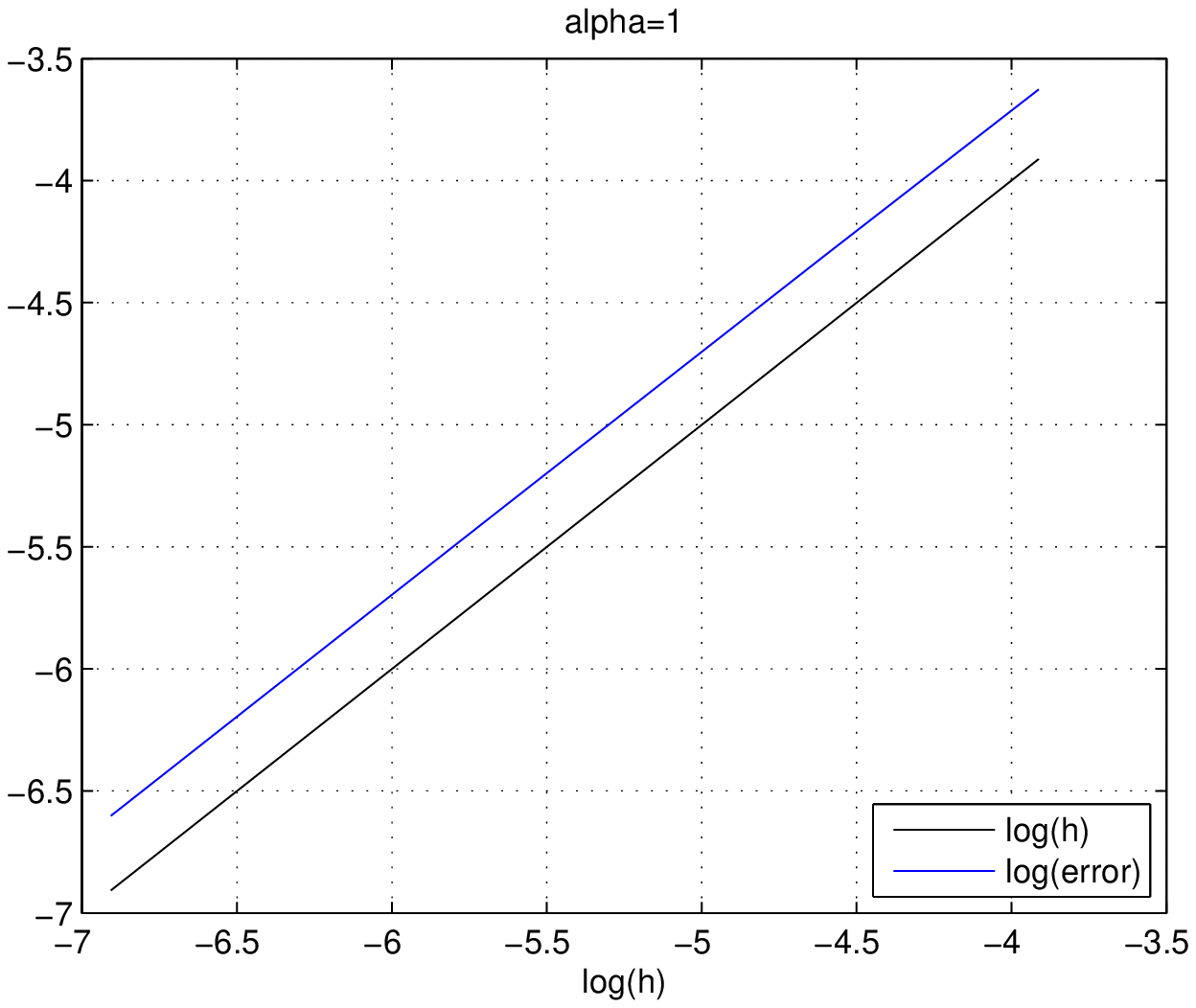} & \includegraphics[width=0.5\textwidth]{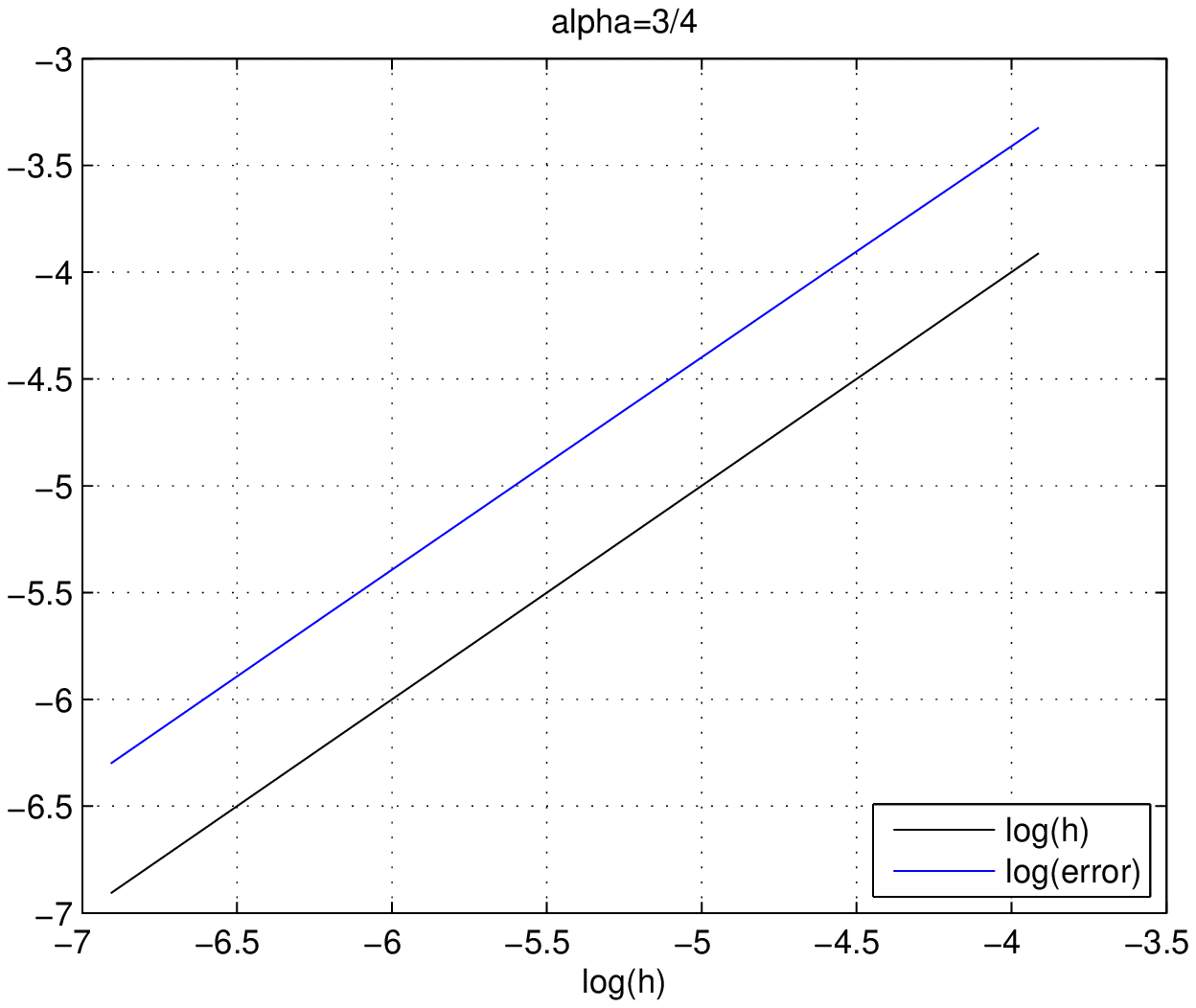}
\end{array}
\end{equation*}
\begin{equation*}
\begin{array}{cc}
\includegraphics[width=0.5\textwidth]{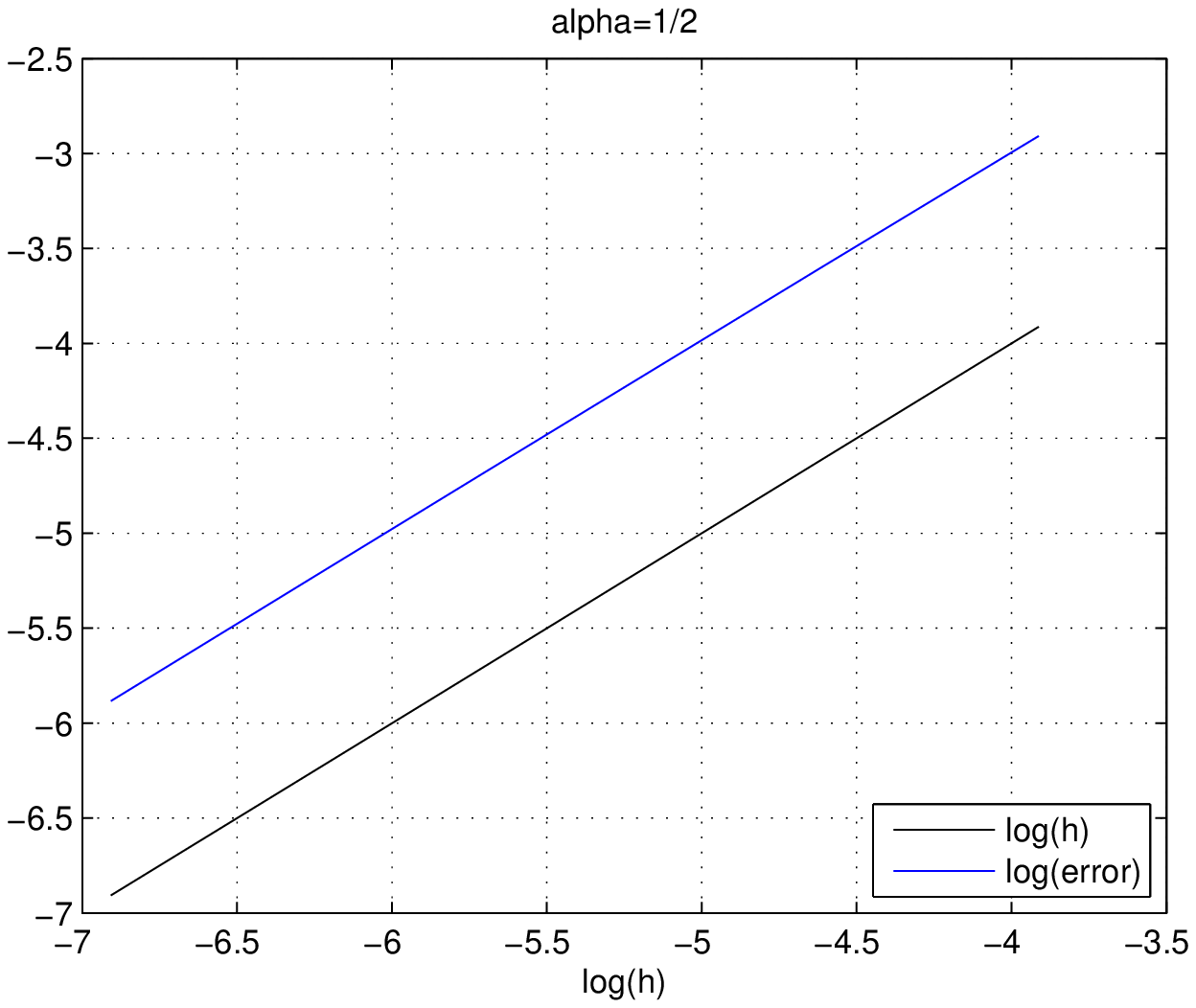} & \includegraphics[width=0.5\textwidth]{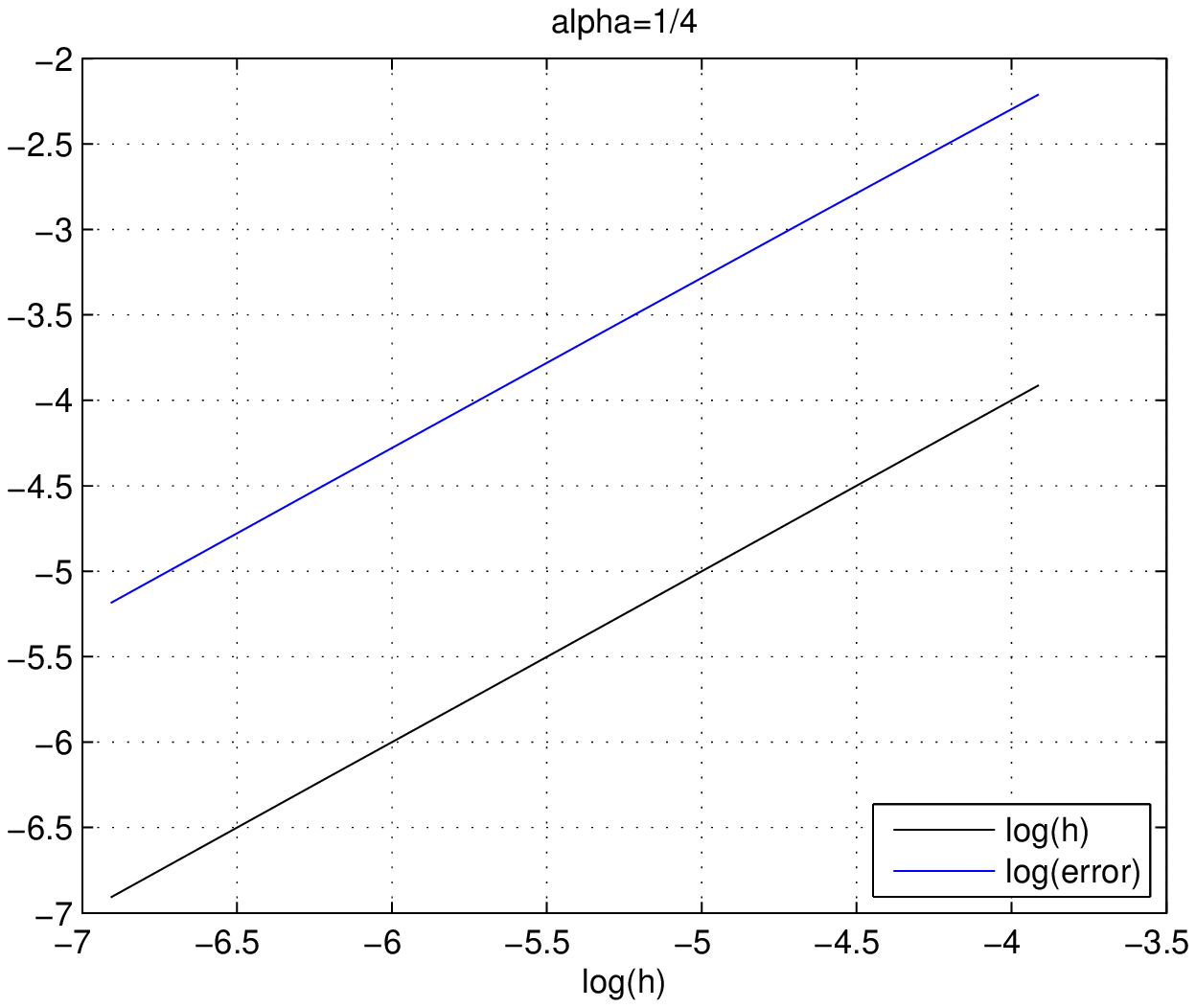}
\end{array}
\end{equation*}
For this example, the convergence seems then obtained for any $\alpha=1$, $3/4$, $1/2$, $1/4$ and still with order $1$. Hence, the graphics obtained in these Sections \ref{section31} and \ref{section32} make us confident with respect to the quality of \eqref{eqpontsystfd} both in the classical and strict fractional cases.

\section{A discrete fractional Noether's theorem}\label{section4}
Fractional Pontryagin's systems \eqref{eqps} are very difficult to solve explicitly, see example in Section \ref{section31}. Consequently, a deviously way in order to get informations on the exact solutions is to derive a constant of motion, \textit{i.e.} functions which are constant on each solution. Such conservation laws allow to obtain many informations in the phase space for example or to integrate the equation by quadrature. In \cite{bour7}, we prove a fractional Noether's theorem giving the existence of an explicit conservation law for fractional Pontryagin's systems \eqref{eqps} exhibiting a symmetry. Let us remind that this result is based on a preliminary result proved by Torres and Frederico in \cite{torr,torr2}. \\

In this section, we study the existence of discrete conservation laws for shifted discrete fractional Pontryagin's systems \eqref{eqpontsystfd}. Precisely, following the same strategy, we introduce the notion of \textit{discrete symmetry} for such systems and prove a discrete fractional Noether's theorem providing an \textit{explicit computable} discrete constant of motion. Let us note that this work is strongly inspired from our study in \cite{bour2} where we have provided a discrete fractional Noether's theorem for discrete fractional Euler-Lagrange equations admitting a discrete symmetry. \\

We first review the definition of a one parameter group of diffeomorphisms:

\begin{definition}
Let $n \in \N^*$. For any real $s$, let $\fonctionsansdef{\phi (s,\cdot)}{\R ^n}{\R ^n}$ be a diffeomorphism. Then, $\Phi = \{ \phi (s,\cdot) \}_{s \in \R}$ is a one parameter group of diffeomorphisms of $\R^n$ if it satisfies:
\begin{enumerate}
\item $\phi (0,\cdot) = Id_{\R ^n}$; 
\item $\forall s,s' \in \R, \; \phi (s,\cdot) \circ \phi (s',\cdot) = \phi (s+s',\cdot) $;
\item $\phi$ is of class $\CC^2$.
\end{enumerate}
\end{definition}

Usual examples of one parameter groups of diffeomorphisms are given by translations and rotations. The action of three one parameter groups of diffeomorphisms on an Hamiltonian allows to define the notion of a discrete symmetry for a shifted discrete fractional Pontryagin's system \eqref{eqpontsystfd}:

\begin{definition}\label{defsymd}
Let $\Phi_i = \{ \phi_i (s,\cdot) \}_{s \in \R}$, for $i=1,2,3$, be three one parameter groups of diffeomorphisms of $\R^d$, $ \R^m$ and $\R^d$ respectively. Let $L$ be a Lagrangian, $f$ be a constraint function and $H$ be the associated Hamiltonian. $H$ is said to be $\cDDM$-invariant under the action of $(\Phi_i)_{i=1,2,3}$ if it satisfies: for any $(\Q,\U,\PP)$ solution of \eqref{eqpontsystfd} and any $s \in \R$
\begin{multline}\label{eq4-1}
H \Big( \phi_1 \big(s,\Q\big), \phi_2 \big(s,\U\big) , \phi_3 \big(s,\sigma^{-1}(\PP) \big) ,\T \Big) - \phi_3 \big(s,\sigma^{-1}(\PP) \big) \cdot \cDDM \Big( \phi_1 \big(s,\Q\big) \Big) \\ = H\big(\Q,\U,\sigma^{-1}(\PP),\T \big) - \sigma^{-1}(\PP) \cdot \cDDM \Q.
\end{multline}
\end{definition}

From this notion, we prove the following Lemma:

\begin{lemma}\label{lemtorrd}
Let $L$ be a Lagrangian, $f$ be a constraint function and $H$ be the associated Hamiltonian. Let us assume that $H$ is $\cDDM$-invariant under the action of three one parameter groups of diffeomorphisms $(\Phi_i)_{i=1,2,3}$. Then, the following equality holds for any solution $(\Q,\U,\PP)$ solution of \eqref{eqpontsystfd}:
\begin{equation}\label{eqlemtorrd}
 \dfrac{\partial \phi_1}{\partial s} (0,\Q) \cdot \sigma^{-1} ( \DDP \PP ) - \cDDM \left( \dfrac{\partial \phi_1}{\partial s} (0,\Q) \right) \cdot \sigma^{-1} (\PP ) = 0.
\end{equation}
\end{lemma}

\begin{proof}
Let us differentiate \eqref{eq4-1} with respect to $s$ and let us invert the operator $ \cDDM $ and $\partial / \partial s $. Taking $s=0$, we finally obtain:
\begin{multline}
\dfrac{\partial H}{\partial x} (\star) \cdot \dfrac{\partial \phi_1}{\partial s}(0,\Q) + \dfrac{\partial H}{\partial v} (\star) \cdot \dfrac{\partial \phi_2}{\partial s}(0,\U) + \dfrac{\partial H}{\partial w} (\star) \cdot \dfrac{\partial \phi_3}{\partial s}\big( 0,\sigma^{-1}(\PP) \big) \\ - \dfrac{\partial \phi_3}{\partial s}\big( 0,\sigma^{-1}(\PP) \big) \cdot \cDDM \Q - \sigma^{-1} (\PP) \cdot \cDDM \left( \dfrac{\partial \phi_1}{\partial s}(0,\Q) \right) = 0,
\end{multline} 
where $\star = \big( \Q,\U,\sigma^{-1}(\PP),\T \big)$. Since $(\Q,\U,\PP)$ is solution of \eqref{eqpontsystfd}, we obtain \eqref{eqlemtorrd}.
\end{proof}

Let us note that this last result corresponds to the discrete version of the result proved by Torres and Frederico in \cite{torr,torr2}. Let us remind that our aim is to provide an explicit discrete constant of motion for shifted discrete fractional Pontryagin's systems \eqref{eqpontsystfd} exhibiting a discrete symmetry. Our result is based on Lemma \ref{lemtorrd} and on the following implication:
\begin{equation}\label{eq4-2}
\forall \G \in \R^{N+1}, \; \Delta^1_- \G = 0 \Longrightarrow \exists c \in \R, \; \forall k=0,\ldots,N, \; G_k = c.
\end{equation}
Namely, if the discrete derivative of $ \G$ vanishes, then $ \G$ is constant. Consequently, our aim is to write the left term of \eqref{eqlemtorrd} as an explicit discrete derivative (\textit{i.e.} as $\Delta^1_-$ of an explicit quantity). In this way, we are going to use a \textit{discrete transfer formula} as it is done in \cite{bour2} for discrete fractional Euler-Lagrange equations admitting a discrete symmetry. \\

Nevertheless, we have first to introduce some square matrices of length $(N+1)$. First, $B_{1} := \text{Id}_{N+1}$ and then, for any $r \in \{ 2,\ldots,N \}$, the square matrices $B_r \in \mathcal{M}_{N+1}$ defined by:
\begin{equation}
\forall i,j=0,\ldots,N, \; (B_r)_{i,j} := \delta_{\{1 \leq i \leq N-1\}} \delta_{\{1 \leq j \leq N-r\}} \delta_{\{0 \leq i-j \leq r-1\}} - \delta_{\{j=0\}} \delta_{\{r \leq i\}},
\end{equation}
where $\delta$ is the Kronecker symbol. Secondly, we define the square matrices $C_r \in \mathcal{M}_{N+1}$ by:
\begin{equation}
\forall r=1,\ldots,N, \; \forall i,j=0,\ldots,N, \; (C_r)_{i,j} := \delta_{\{r \leq i\}} \delta_{\{j=0\}} .
\end{equation}
Finally, we define the square matrices $A_r \in \mathcal{M}_{N+1}$ by:
\begin{equation}
\forall r=1,\ldots,N, \; A_r := \alpha_r B_r + \beta^\alpha_r C_r ,
\end{equation}
where $\beta^\alpha_r = \sum_{k=0}^{r} \alpha_k $. Examples of matrices $A_r \in \mathcal{M}_{N+1}$ for $N=5$ are given in Appendix \ref{appAd}.

\begin{lemma}[Discrete transfer formula]\label{lemdtransform}
Let $\G^1$, $\G^2 \in \RN$ satisfying $G^2_N = 0$. Then, the following equality holds:
\begin{equation}\label{eq4-3}
 \G^1 \cdot \sigma^{-1} ( \DDP \G^2 ) - (\cDDM \G^1) \cdot \sigma^{-1} (\G^2 ) = h^{1-\alpha} \Delta^1_- \Big[ \di \sum_{r=1}^N A_r \times \big( \G^1 \cdot \sigma^{r-1} (\G^2 ) \big) \Big].
\end{equation}
\end{lemma}

\begin{proof} See Appendix \ref{appAd}. \end{proof}

Consequently, combining Lemmas \ref{lemtorrd} and \ref{lemdtransform}, we prove:

\begin{theorem}[Discrete fractional Noether's theorem]\label{thmdfnoether}
Let $L$ be a Lagrangian, $f$ be a constraint function and $H$ be the associated Hamiltonian. Let us assume that $H$ is $\cDDM$-invariant under the action of three one parameter groups of diffeomorphisms $(\Phi_i)_{i=1,2,3}$. Then, the following equality holds for any solution $(\Q,\U,\PP)$ of \eqref{eqpontsystfd}:
\begin{equation}
\Delta^1_- \left[ \sum_{r=1}^N A_r \times \left( \dfrac{\partial \phi_1}{\partial s} (0,\Q)  \cdot \sigma^{r-1} (\PP) \right) \right] = 0.
\end{equation}
\end{theorem}
According to Equation \eqref{eq4-2}, this theorem provides a discrete constant of motion for any shifted discrete fractional Pontryagin's systems \eqref{eqpontsystfd} exhibiting a discrete symmetry. Moreover, this discrete conservation law is not only \textit{explicit} but also \textit{computable} in a finite number of steps. Let us see a concrete example:

\begin{example}\label{ex}
Let us consider $d = m = 2$, the following quadratic Lagrangian and the following linear constraint function:
\begin{equation}
\fonction{L}{\R^2 \times \R^2 \times [a,b]}{\R}{(x,v,t)}{(\Vert x \Vert^2 + \Vert v \Vert^2 )/2} \quad \text{and} \quad \fonction{f}{\R^2 \times \R^2 \times [a,b]}{\R^2}{(x,v,t)}{x+v.}
\end{equation}
Then, we consider the three one parameter groups of diffeomorphisms given by the following rotations:
\begin{equation}
\fonction{\phi_i}{\R \times \R^2}{\R^2}{(s,x_1,x_2)}{\left( \begin{array}{cc} \cos (s \theta_i) & - \sin (s \theta_i) \\ \sin (s \theta_i) & \cos (s \theta_i)  \end{array} \right) \left( \begin{array}{c} x_1 \\ x_2 \end{array} \right),}
\end{equation}
for $i=1,2,3$ and where $\theta_1$, $\theta_2 \in \R$ and $\theta_3 = - \theta_1$. With these parameters, one can prove that the Hamiltonian $H$ associated to $L$ and $f$ is $\cDDM$-invariant under the action of $(\Phi_i)_{i=1,2,3}$. Consequently, the fractional Pontryagin's system \eqref{eqpontsystfd} admits a symmetry and then admits an explicit discrete conservation law given by the discrete fractional Noether's Theorem \ref{thmdfnoether}. \\

We choose $A=(1,2)$, $N=100$ and $\theta_1 = \theta_2 = - \theta_3 = 1$. Let us compute \eqref{eqpontsystfd} for $\alpha=1$, $3/4$, $1/2$, $1/4$. Then, we denote $\Q = (\Q^1,\Q^2)$ and $\PP = (\PP^1,\PP^2)$ the discrete solutions obtained and we denote $\G = \partial \phi_1 / \partial s (0,\Q) = (-\Q^2,\Q^1)$. We are then interested in the value of:
\begin{equation}\label{eq4-4}
\sum_{r=1}^N A_r \times \left( \G  \cdot \sigma^{r-1} (\PP) \right).
\end{equation}
Let us see the graphics obtained by the computation of \eqref{eqpontsystfd} and by the computation of the quantity given in Equation \eqref{eq4-4} for $\alpha=1$, $3/4$, $1/2$, $1/4$:
\begin{equation*}
\begin{array}{cc}
\includegraphics[width=0.5\textwidth]{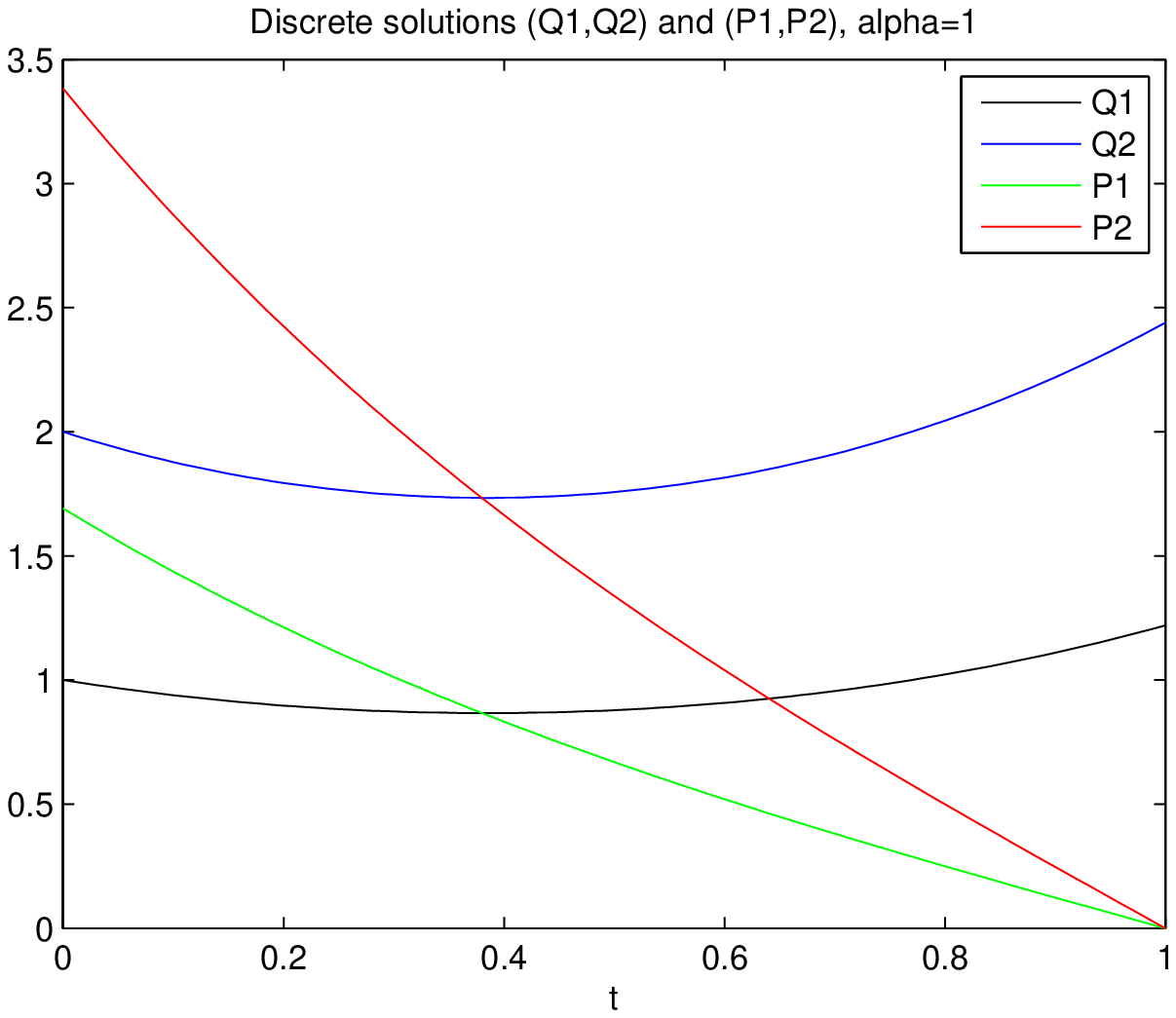} & \includegraphics[width=0.5\textwidth]{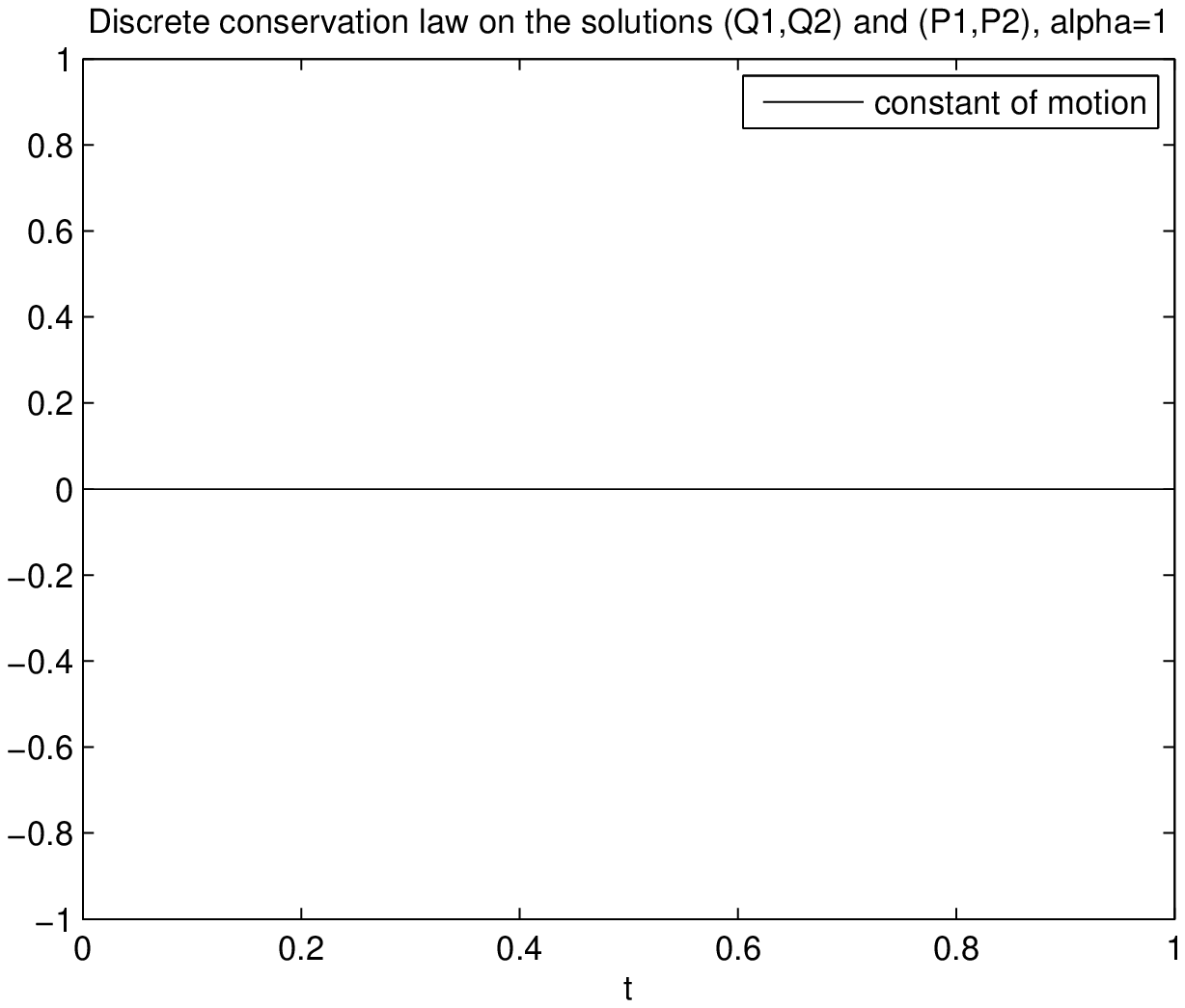}
\end{array}
\end{equation*}
\begin{equation*}
\begin{array}{cc}
\includegraphics[width=0.5\textwidth]{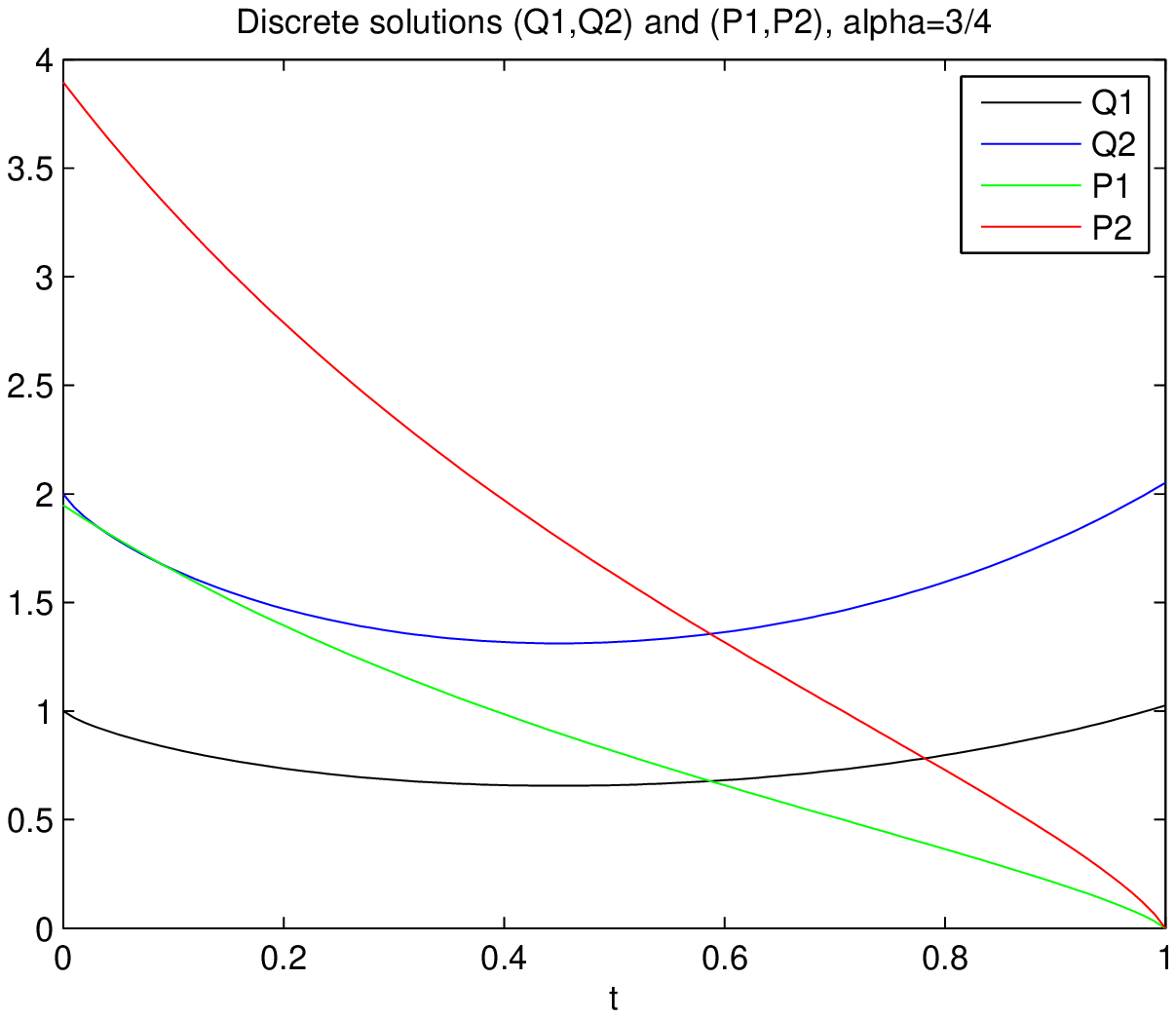} & \includegraphics[width=0.5\textwidth]{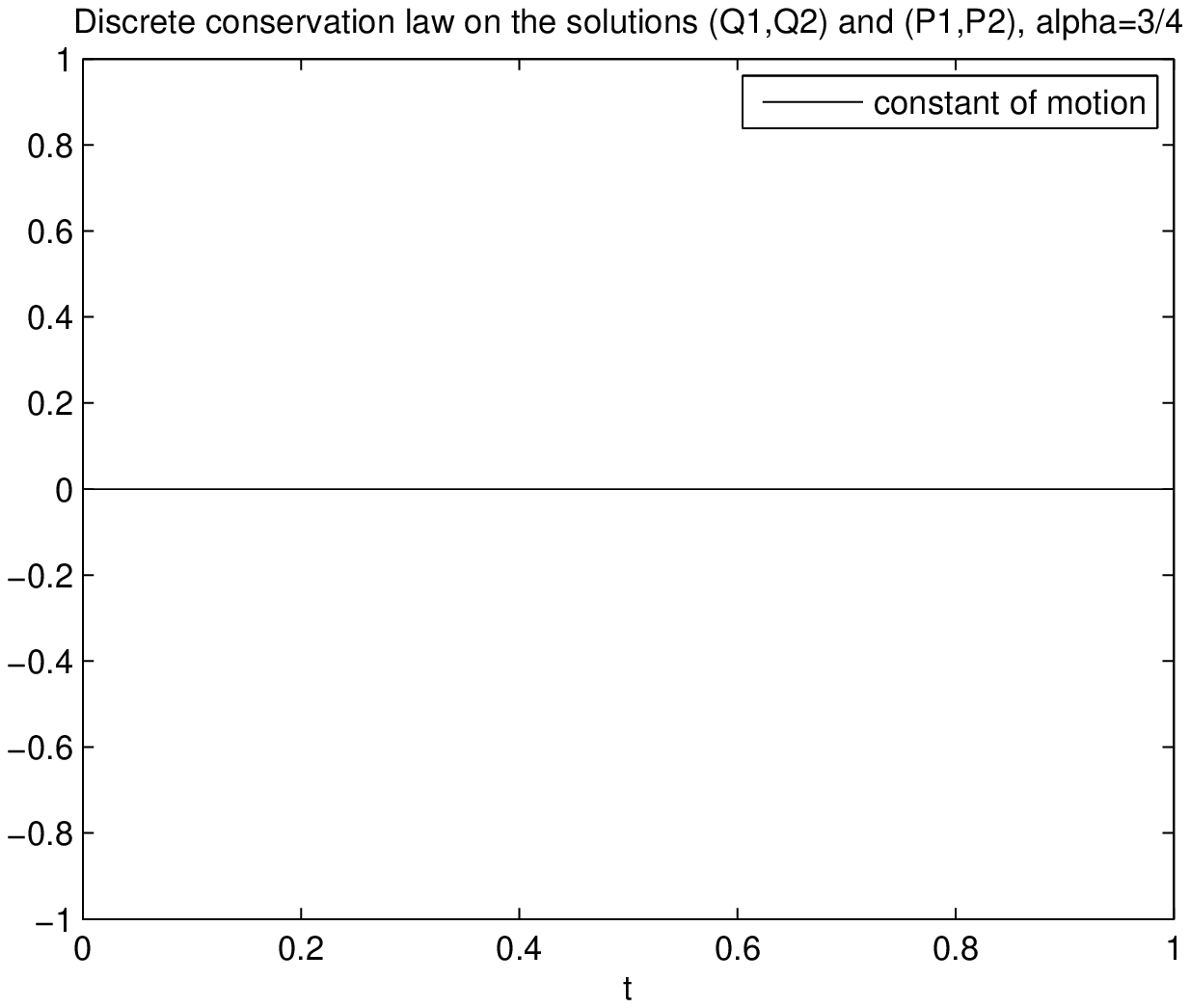}
\end{array}
\end{equation*}
\begin{equation*}
\begin{array}{cc}
\includegraphics[width=0.5\textwidth]{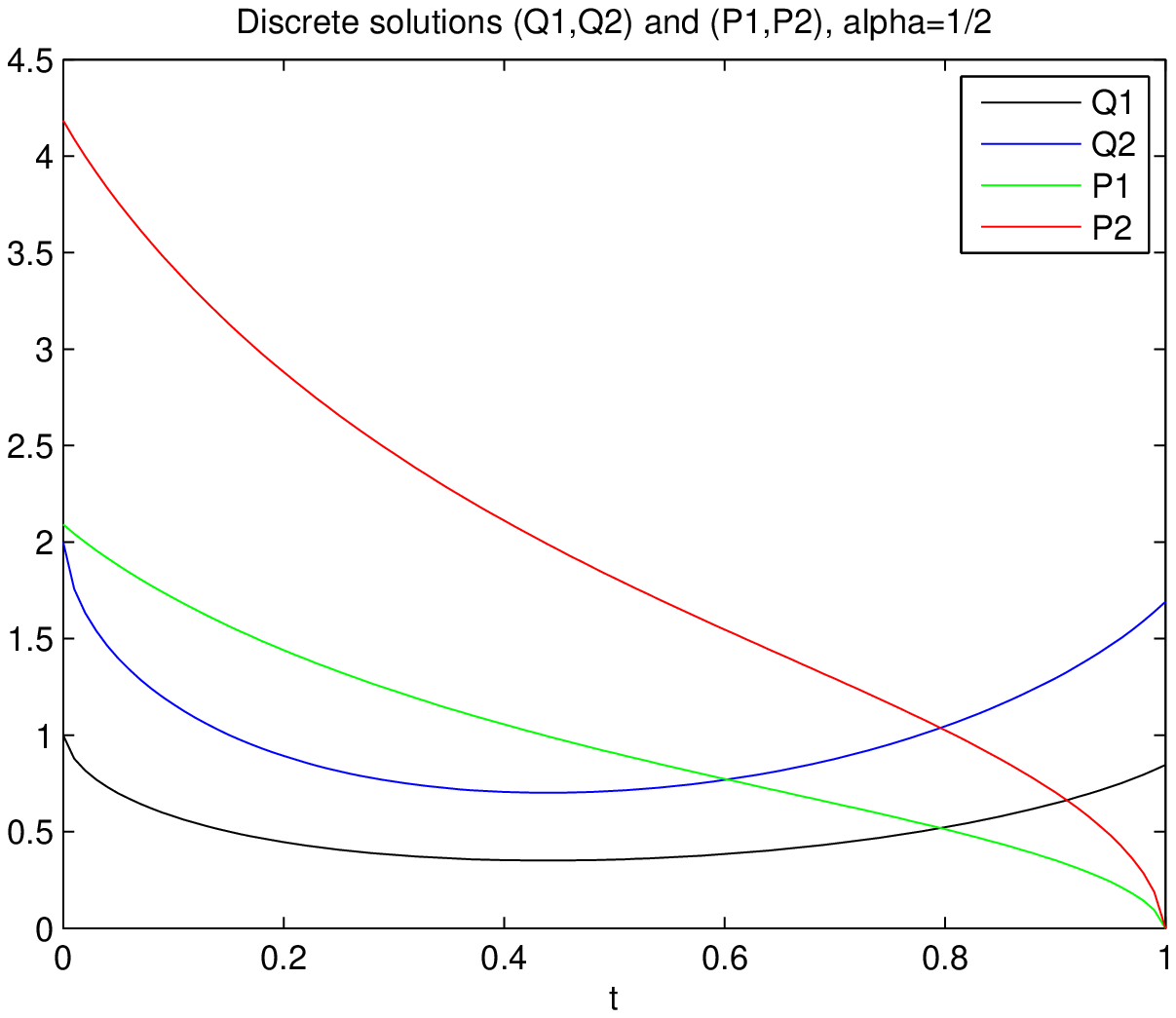} & \includegraphics[width=0.5\textwidth]{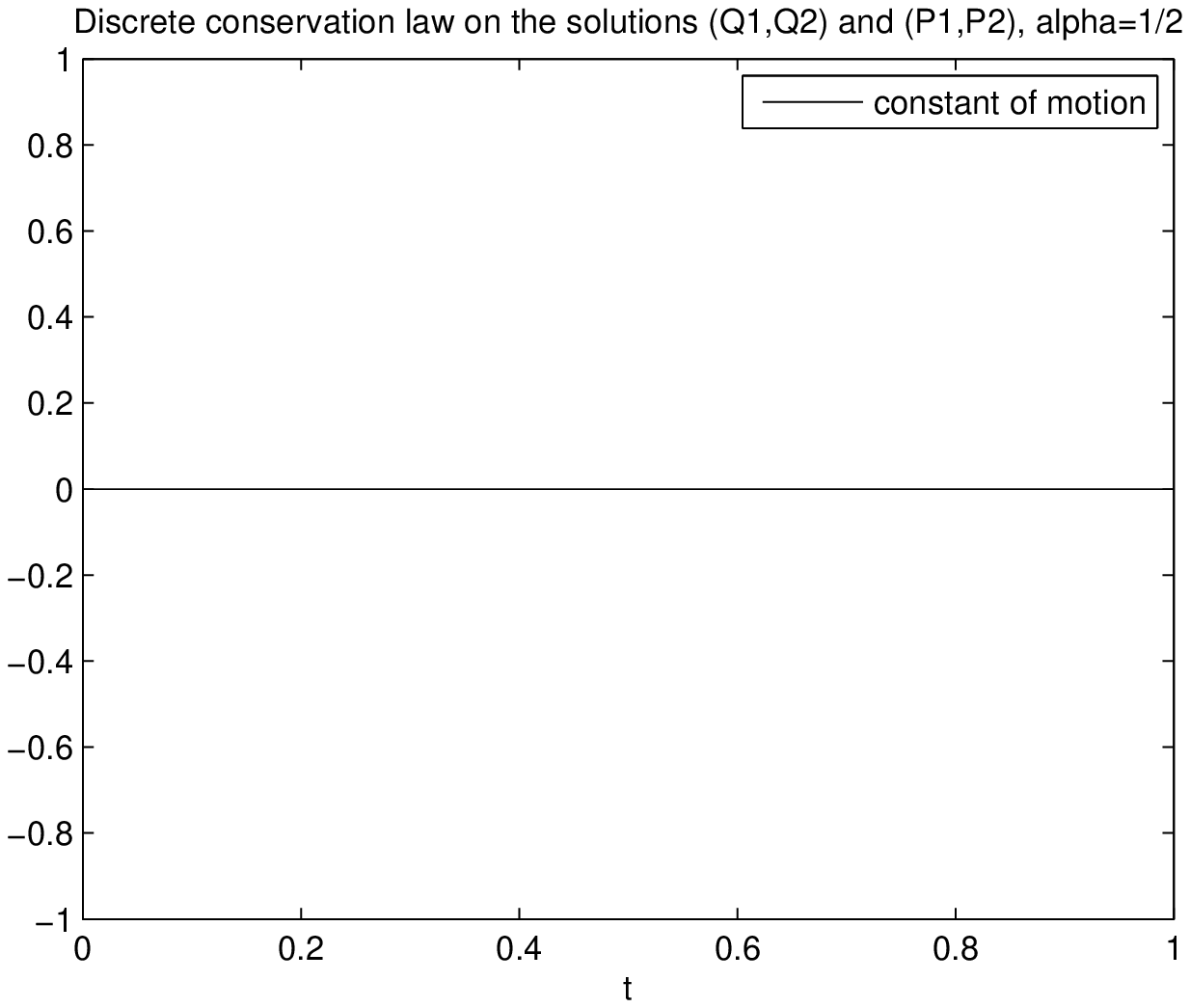}
\end{array}
\end{equation*}
\begin{equation*}
\begin{array}{cc}
\includegraphics[width=0.5\textwidth]{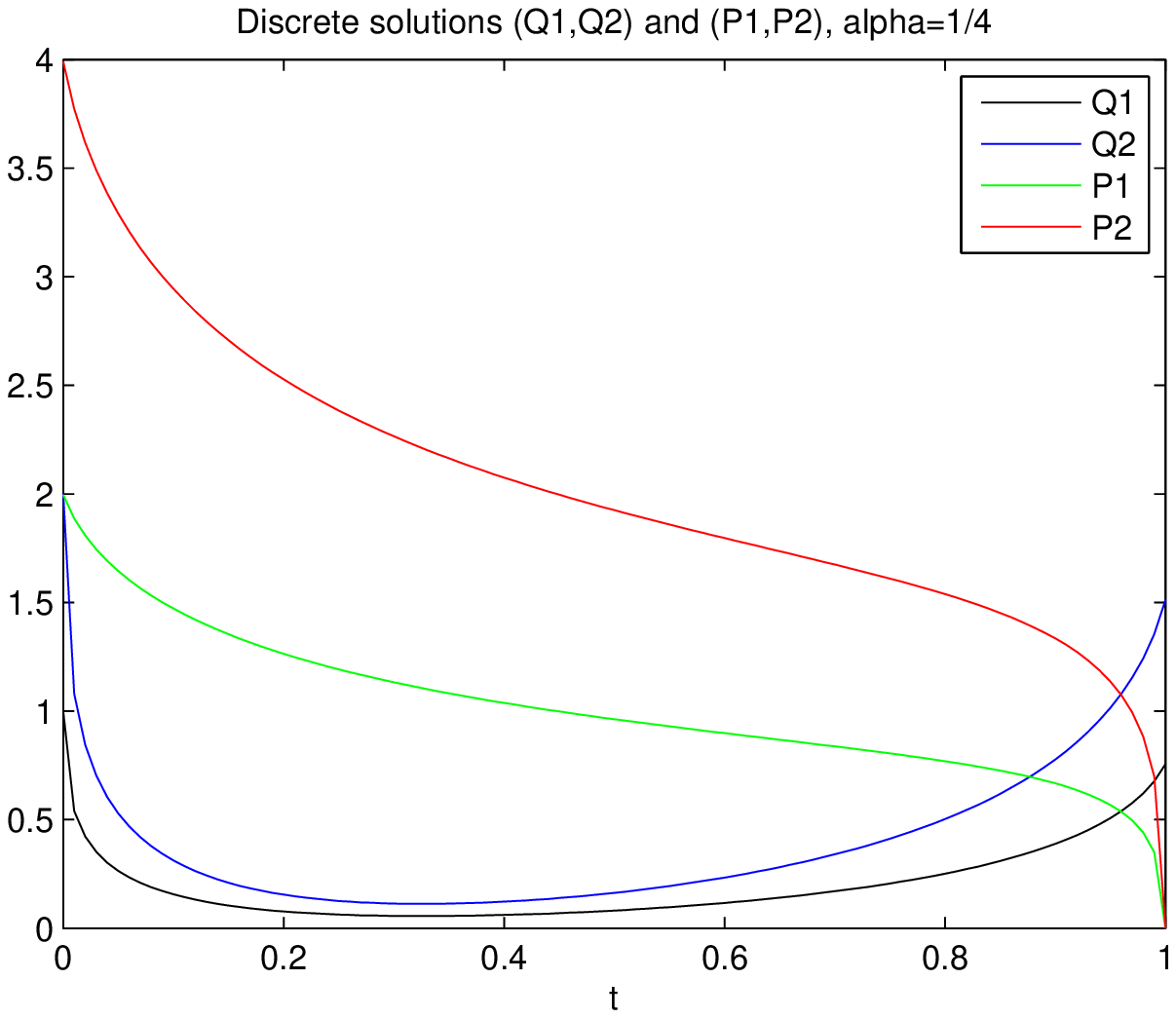} & \includegraphics[width=0.5\textwidth]{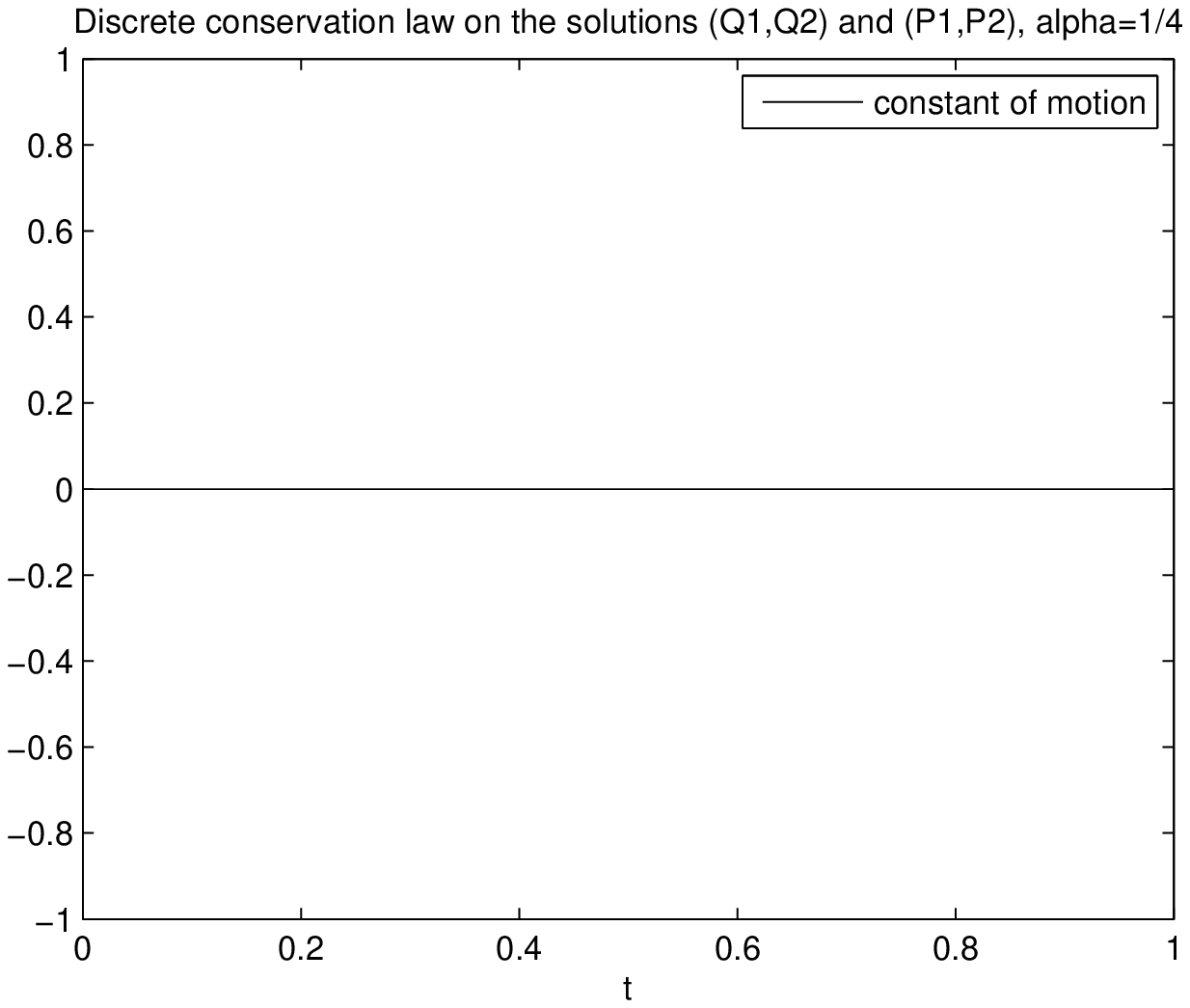}
\end{array}
\end{equation*}
As expected from Theorem \ref{thmdfnoether}, we obtain discrete constants of motion for this discrete fractional Pontryagin's system \eqref{eqpontsystfd} admitting a discrete symmetry and for any $\alpha=1$, $3/4$, $1/2$, $1/4$. In this specific example, the constant obtained is zero.
\end{example}

\appendix

\section{Appendix A}\label{appA}

\subsection{Result of stability of order $1$}\label{appAa}
In this section, we use the notations and definitions given in Sections \ref{section23} and \ref{section24}. Let us prove the following Lemma:
\begin{lemma}\label{lemappAa}
Let $\U$, $\bar{\U} \in (\R^m)^{N+1}$. Then, there exists a constant $C_1 \geq 0$ such that:
\begin{equation}
\forall \vert \eps \vert <1, \; \forall k=0,\ldots,N, \; \Vert Q^{\U+\eps \bar{\U},\alpha}_k - Q^{\U,\alpha}_k \Vert \leq C_1 \vert \eps \vert.
\end{equation}
\end{lemma}

\begin{proof}
First, let us prove by induction that:
\begin{equation}\label{eqAa-1}
\forall k=0,\ldots,N, \; \exists R_k \geq 0, \; \forall \vert \eps \vert <1, \; \Vert Q^{\U+\eps \bar{\U},\alpha}_k - Q^{\U,\alpha}_k \Vert \leq R_k \vert \eps \vert .
\end{equation}
For $k=0$, let us take $R_0 = 0$ since $Q^{\U+\eps \bar{\U},\alpha}_0 = Q^{\U,\alpha}_0 = A$ for any $\abs{\eps} <1$. Let $k \in \{ 1,\ldots,N\}$ and let us assume that the result \eqref{eqAa-1} is satisfied for any $n=0,\ldots,k-1$. Let us prove that the result \eqref{eqAa-1} is then satisfied for $n=k$. Since $\Qa$ (resp. $\Q^{\U+\eps \bar{\U},\alpha}$) is solution of \eqref{eqdcpq} associated to $\U$ (resp. to $\U + \eps \bar{\U}$), we have: 
\begin{equation}
Q^{\U,\alpha}_k = h^\alpha f(Q^{\U,\alpha}_k,U_k,t_k)+A - \di \sum_{r=1}^{k} \alpha_{r} ( Q^{\U,\alpha}_{k-r} - A)
\end{equation}
and for any $\abs{\eps} < 1$:
\begin{equation}
Q^{\U+\eps \bar{\U},\alpha}_k = h^\alpha f(Q^{\U+\eps \bar{\U},\alpha}_k,U_k+\eps \bar{U}_k,t_k)+A - \di \sum_{r=1}^{k} \alpha_{r} ( Q^{\U+\eps \bar{\U},\alpha}_{k-r} - A).
\end{equation}
Consequently, for any $\abs{\eps} < 1$:
\begin{multline}
\Vert Q^{\U+\eps \bar{\U},\alpha}_k - Q^{\U,\alpha}_k \Vert \leq h^\alpha \Vert f(Q^{\U+\eps \bar{\U},\alpha}_k,U_k+\eps \bar{U}_k,t_k) - f(Q^{\U,\alpha}_k,U_k,t_k) \Vert \\ + \di \sum_{r=1}^{k} \vert \alpha_{r} \vert \Vert Q^{\U+\eps \bar{\U},\alpha}_{k-r} - Q^{\U,\alpha}_{k-r} \Vert,
\end{multline}
and then, with the induction hypothesis, we have for any $\abs{\eps} < 1$:
\begin{multline}
\Vert Q^{\U+\eps \bar{\U},\alpha}_k - Q^{\U,\alpha}_k \Vert \leq h^\alpha \Vert f(Q^{\U+\eps \bar{\U},\alpha}_k,U_k+\eps \bar{U}_k,t_k) - f(Q^{\U,\alpha}_k,U_k+\eps \bar{U}_k,t_k) \Vert \\ + h^\alpha \Vert f(Q^{\U,\alpha}_k,U_k+\eps \bar{U}_k,t_k) - f(Q^{\U,\alpha}_k,U_k,t_k) \Vert + \di \sum_{r=1}^{k} \vert \alpha_{r} \vert R_{k-r} \vert \eps \vert .
\end{multline}
Finally, using Condition \eqref{condf} and a Taylor's expansion of order $1$ with explicit remainder, we prove:
\begin{multline}
\forall \abs{\eps} < 1, \; \Vert Q^{\U+\eps \bar{\U},\alpha}_k - Q^{\U,\alpha}_k \Vert \leq  h^\alpha M \Vert Q^{\U+\eps \bar{\U},\alpha}_k - Q^{\U,\alpha}_k \Vert \\ + h^\alpha \abs{\eps} \left\Vert \dfrac{\partial f}{\partial v}(Q^{\U,\alpha}_k ,\xi^{\eps},t_k) \times \bar{U}_k \right\Vert + \di \sum_{r=1}^{k} \vert \alpha_{r} \vert R_{k-r} \vert \eps \vert,
\end{multline}
where $\xi^{\eps} \in [U_k,U_k + \eps \bar{U}_k] \subset [-M_1,M_1]^m $ with $M_1 \geq 0$ independent of $\abs{\eps} <1$. Hence, since $\partial f / \partial v $ is continuous, we can conclude that there exists $M_2 \geq 0$ such that:
\begin{equation}
\forall \abs{\eps} <1 ,\; \left\Vert \dfrac{\partial f}{\partial v}(Q^{\U,\alpha}_k ,\xi^{\eps},t_k) \times \bar{U}_k \right\Vert \leq M_2 .
\end{equation}
Consequently, since $h^\alpha M<1$ from Condition \eqref{condh}, we have:
\begin{equation}
\forall \abs{\eps} <1, \; \Vert Q^{\U+\eps \bar{\U},\alpha}_k - Q^{\U,\alpha}_k \Vert \leq \dfrac{\vert \eps \vert}{1-h^\alpha M} \left( h^\alpha M_2 + \di \sum_{r=1}^{k} R_{k-r} \vert \alpha_{r} \vert \right).
\end{equation}
We then define $R_k := ( h^\alpha M_2 + \sum_{r=1}^{k} R_{k-r} \vert \alpha_{r} \vert )/(1-h^\alpha M)$ independent $\abs{\eps} <1 $ which concludes the induction. To complete the proof, we have just to define $C_1 = \max \{ R_k, \; k=0,\ldots,N \}$.
\end{proof}

\subsection{Result of stability of order $2$}\label{appAb}
In this section, we use the notations and definitions given in Sections \ref{section23} and \ref{section24}. We prove the following Lemma with the help of Lemma \ref{lemappAa}:
\begin{lemma}\label{lemappAb}
Let $\U$, $\bar{\U} \in (\R^m)^{N+1}$. Then, there exists a constant $C \geq 0$ such that:
\begin{equation}
\forall \vert \eps \vert < 1, \; \forall k=0,\ldots,N , \; \Vert Q^{\U+\eps \bar{\U},\alpha}_k - Q^{\U,\alpha}_k - \eps \bar{Q}_k \Vert  \leq C \eps^2 ,
\end{equation}
where $\bar{\Q}$ is the unique solution of the following linearised discrete fractional Cauchy problem:
\begin{equation}\tag{LCP${}^\alpha_{\bar{\Q}}$}
 \left\lbrace \begin{array}{l}
 		\cDDM \bar{\Q} = \dfrac{\partial f}{\partial x} (\Qa,\U,\T) \times \bar{\Q} + \dfrac{\partial f}{\partial v} (\Qa,\U,\T) \times \bar{\U} \\[10pt]
 		\bar{Q}_0 = 0.
        \end{array}
\right. 
\end{equation}
Its existence and its uniqueness are provided by Theorem \ref{thmdfcl} and Conditions \eqref{condf} and \eqref{condh}.
\end{lemma}

\begin{proof}
We proceed in the same manner that for Lemma \ref{lemappAa}. Let us prove by induction that:
\begin{equation}\label{eqAb-1}
\forall k=0,\ldots,N , \; \exists R_k \geq 0, \; \forall \vert \eps \vert <1, \; \Vert Q^{\U+\eps \bar{\U},\alpha}_k -  Q^{\U,\alpha}_k - \eps \bar{Q}_k \Vert \leq R_k \eps^2.
\end{equation}
For $k=0$, let us take $R_0 = 0$ since $Q^{\U+\eps \bar{\U},\alpha}_0 = Q^{\U,\alpha}_0=A$ for any $\abs{\eps} < 1$ and $\bar{Q}_0 = 0$. Let $k \in \{1,\ldots,N \}$ and let us assume that the result \eqref{eqAb-1} is satisfied for any $n=0,\ldots,k-1$. Let us prove that the result \eqref{eqAb-1} is then satisfied for $n=k$. Since $\Qa$ (resp. $\Q^{\U+\eps \bar{\U},\alpha}$) is solution of \eqref{eqldcpq} associated to $\U$ (resp. to $\U + \eps \bar{\U}$) and since $\bar{\Q}$ is solution of \eqref{eqldcpq}, we have with a Taylor's expansion of order $2$ with explicit remainder:
\begin{multline}\label{eqAb-2}
\forall \abs{\eps} < 1, \; Q^{\U+\eps \bar{\U},\alpha}_k - Q^{\U,\alpha}_k - \eps \bar{Q}_k = h^\alpha \dfrac{\partial f}{\partial x} (Q^{\U,\alpha}_k,U_k,t_k) \times ( Q^{\U+\eps \bar{\U},\alpha}_k - Q^{\U,\alpha}_k - \eps \bar{Q}_k ) \\ + h^\alpha \left( \dfrac{1}{2} \nabla^2 f(\xi_1^\eps,\xi_2^\eps,t_k) (Q^{\U+\eps \bar{\U},\alpha}_k - Q^{\U,\alpha}_k, \eps \bar{U}_k,0)^2 \right) - \di \sum_{r=1}^k \alpha_r ( Q^{\U+\eps \bar{\U},\alpha}_{k-r} - Q^{\U,\alpha}_{k-r} - \eps \bar{Q}_{k-r} ),
\end{multline}
where:
\begin{itemize}
\item $\xi_1^\eps \in [Q^{\U,\alpha}_{k},Q^{\U+\eps \bar{\U},\alpha}_{k}] \subset [\Vert Q^{\U,\alpha}_{k} \Vert - C_1,\Vert Q^{\U,\alpha}_{k} \Vert + C_1]$ from Lemma \ref{lemappAa}. Then, $\xi_1^\eps \in (-M_1,M_1]^d$ with $M_1 \geq 0$ independent of $\abs{\eps} < 1$;
\item $\xi_2^\eps \in [U_{k},U_{k}+ \eps \bar{U}_k] \subset [-M_2,M_2]^m$ with $M_2 \geq 0$ independent of $\abs{\eps} < 1$.
\end{itemize}
Since $\nabla^2 f (\cdot,\cdot,t_k)$ is continuous, we conclude that there exists $M_3 \geq 0$ such that:
\begin{multline}
\forall \abs{\eps} < 1, \; \left\Vert \dfrac{1}{2} \nabla^2 f(\xi_1^\eps,\xi_2^\eps,t_k) (Q^{\U+\eps \bar{\U},\alpha}_k - Q^{\U,\alpha}_k, \eps \bar{U}_k,0 )^2 \right\Vert \\ \leq M_3 ( \Vert Q^{\U+\eps \bar{\U},\alpha}_k - Q^{\U,\alpha}_k \Vert^2 + 2 \Vert Q^{\U+\eps \bar{\U},\alpha}_k - Q^{\U,\alpha}_k \Vert \Vert \eps \bar{U}_k \Vert + \Vert \eps \bar{U}_k \Vert^2).
\end{multline}
Hence, from Lemma \ref{lemappAa}, there exists $M_4 \geq 0$ such that:
\begin{equation}\label{eqAb-3}
\forall \abs{\eps} < 1, \; \left\Vert \dfrac{1}{2} \nabla^2 f(\xi_1^\eps,\xi_2^\eps,t_k) (Q^{\U+\eps \bar{\U},\alpha}_k - Q^{\U,\alpha}_k, \eps \bar{U}_k,0 )^2 \right\Vert \leq M_4 \eps^2.
\end{equation}
From Equality \eqref{eqAb-2} and Condition \eqref{condf}, Inequality \eqref{eqAb-3} and the induction hypothesis, we obtain:
\begin{multline}
\forall \abs{\eps} < 1, \; \Vert Q^{\U+\eps \bar{\U},\alpha}_k - Q^{\U,\alpha}_k - \eps \bar{Q}_k \Vert \leq 2 h^\alpha M \Vert Q^{\U+\eps \bar{\U},\alpha}_k - Q^{\U,\alpha}_k - \eps \bar{Q}_k \Vert  \\ + \di \sum_{r=1}^k \vert \alpha_r \vert R_{k-r} \eps^2 + h^\alpha M_4 \eps^2.
\end{multline}
Finally, since $2 h^\alpha M <1$ from Condition \eqref{condh}, we have:
\begin{equation}
\forall \abs{\eps} < 1,\; \Vert Q^{\U+\eps \bar{\U},\alpha}_k - Q^{\U,\alpha}_k - \eps \bar{Q}_k \Vert \leq \dfrac{\eps^2}{1-2 h^\alpha M} \left( h^\alpha M_4 + \di \sum_{r=1}^{k} R_{k-r} \vert \alpha_{r} \vert \right).
\end{equation}
We then define $R_k := ( h^\alpha M_4 + \sum_{r=1}^{k} R_{k-r} \vert \alpha_{r} \vert )/(1-2 h^\alpha M)$ which concludes the induction. In order to complete the proof, we just have to define $C := \max \{ R_k, \; k=0,\ldots,N \}$.
\end{proof}

\subsection{Proof of Lemma \ref{lem2}}\label{appAc}
In this section, we prove Lemma \ref{lem2} and consequently, we use notations and definitions given in Sections \ref{section23} and \ref{section24}. \\

Let $\U$, $\bar{\U} \in (\R^m )^{N+1}$ and $\bar{\Q} \in ( \R^d )^{N+1}$ the unique solution of \eqref{eqldcpq}. From Lemma \ref{lemappAb}, we have:
\begin{equation}\label{eqAc-1}
\forall k=0,\ldots,N , \; \forall \vert \eps \vert < 1, \; Q^{\U+\eps \bar{\U},\alpha}_k = Q^{\U,\alpha}_k + \eps \bar{Q}_k + H^{\eps}_k,
\end{equation}
where for any $k=0,\ldots,N$ and for any $\abs{\eps} <1$, $\Vert H^{\eps}_k \Vert \leq C \eps^2$. In particular, there exists $M_1 \geq 0$ such that:
\begin{equation}
\forall k=0,\ldots,N , \; \forall \vert \eps \vert < 1, \; [Q^{\U,\alpha}_k,Q^{\U+\eps \bar{\U},\alpha}_k] \subset [-M_1,M_1]^d.
\end{equation}
In the same way, there exists $M_2 \geq 0$ such that:
\begin{equation}
\forall k=0,\ldots,N , \; \forall \vert \eps \vert < 1, [U_k,U_k+\eps \bar{U}_k] \subset [-M_2,M_2]^m.
\end{equation}
We have:
\begin{equation}\label{eqAc-2}
\forall \vert \eps \vert < 1, \; \LL^\alpha_h (\U+\eps \bar{\U})-\LL^\alpha_h(\U) = h \di \sum_{k=1}^N \Big[ L(Q^{\U+\eps \bar{\U},\alpha}_k,U_k+\eps \bar{U}_k,t_k) - L(Q^{\U,\alpha}_k,U_k,t_k) \Big] .
\end{equation}
With a Taylor's expansion of order $2$ with explicit remainder, we have for any $\vert \eps \vert < 1$ and any $k=0,\ldots,N$:
\begin{multline}\label{eqAc-3}
\Big\vert L(Q^{\U+\eps \bar{\U},\alpha}_k,U_k+\eps \bar{U}_k,t_k) - L(Q^{\U,\alpha}_k,U_k,t_k) - \eps \dfrac{\partial L}{\partial x} (Q^{\U,\alpha}_k,U_k,t_k) \cdot \bar{Q}_k - \eps \dfrac{\partial L}{\partial v} (Q^{\U,\alpha}_k,U_k,t_k) \cdot \bar{U}_k \Big\vert \\ \leq \Big\vert \dfrac{\partial L}{\partial x} (Q^{\U,\alpha}_k,U_k,t_k) \cdot H^{\eps}_k \Big\vert  + \Big\vert \dfrac{1}{2} \nabla^2 L (\xi^\eps_1,\xi^\eps_2,t_k ) (\eps \bar{Q}_k+H^\eps_k,\eps \bar{U}_k,0\big)^2 \Big\vert,
\end{multline}
where $\xi^\eps_1 \in [Q^{\U,\alpha}_k,Q^{\U+\eps \bar{\U},\alpha}_k] \subset [-M_1,M_1]^d$ and $\xi^\eps_2 \in [U_k,U_k+\eps \bar{U}_k] \subset [-M_2,M_2]^m$. Since $L$ is of class $\CC^2$, we obtain easily that there exists $M_3 \geq 0$ such that for any $\vert \eps \vert < 1$ and any $k=0,\ldots,N$: 
\begin{multline}\label{eqAc-4}
\Big\vert L(Q^{\U+\eps \bar{\U},\alpha}_k,U_k+\eps \bar{U}_k,t_k) - L(Q^{\U,\alpha}_k,U_k,t_k) \\ - \eps \dfrac{\partial L}{\partial x} (Q^{\U,\alpha}_k,U_k,t_k) \cdot \bar{Q}_k - \eps \dfrac{\partial L}{\partial v} (Q^{\U,\alpha}_k,U_k,t_k) \cdot \bar{U}_k \Big\vert \leq M_3 \eps^2.
\end{multline}
Consequently, we have for any $0 < \vert \eps \vert < 1$ and any $k=0,\ldots,N$:
\begin{multline}
\left\vert \dfrac{L(Q^{\U+\eps \bar{\U},\alpha}_k,U_k+\eps \bar{U}_k,t_k) - L(Q^{\U,\alpha}_k,U_k,t_k)}{\eps} \right. \\ \left.  - \dfrac{\partial L}{\partial x} (Q^{\U,\alpha}_k,U_k,t_k) \cdot \bar{Q}_k -  \dfrac{\partial L}{\partial v} (Q^{\U,\alpha}_k,U_k,t_k) \cdot \bar{U}_k \right\vert \leq M_3 \eps.
\end{multline}
Hence:
\begin{equation}
\lim\limits_{\eps \rightarrow 0} \dfrac{\LL^\alpha_h (\U+\eps \bar{\U})-\LL^\alpha_h(\U)}{\eps} =  h \di \sum_{k=1}^N \left[ \dfrac{\partial L}{\partial x} (Q^{\U,\alpha}_k,U_k,t_k) \cdot \bar{Q}_k + \dfrac{\partial L}{\partial v} (Q^{\U,\alpha}_k,U_k,t_k) \cdot \bar{U}_k \right].
\end{equation}
The proof is completed.

\subsection{Proof of Lemma \ref{lemdtransform}}\label{appAd} 
In this section, we use the notations and definitions given in Section \ref{section4}. Let us prove Lemma \ref{lemdtransform}. \\

Let $\G^1$, $\G^2 \in \RN$ satisfying $G^2_N = 0$. First, let us denote for any $k=1,\ldots,N$:
\begin{equation}
X_k := h^\alpha \Big[ G^1_k \cdot \sigma^{-1} ( \DDP \G^2 )_k - (\cDDM \G^1)_k \cdot \sigma^{-1} (\G^2 )_k \Big].
\end{equation}
Our aim is to write $\boldsymbol{X}$ as an explicit discrete derivative (\textit{i.e.} as $\Delta^1_-$ of an explicit quantity). We have for any $k=1,\ldots,N$:
\begin{equation}
X_k =  G^1_k \cdot \left( \di \sum_{r=0}^{N+1-k} \alpha_r G^2_{k+r-1} \right) - \left( \di \sum_{r=0}^k \alpha_r (G^1_{k-r} - G^1_0 ) \right) \cdot G^2_{k-1} = \alpha_1 h (\Delta^1_- \G^1 \cdot \G^2 )_k  + Y_k + Z_k,
\end{equation}
where for any $k=1,\ldots,N$:
\begin{equation}
Y_k := \left( \di \sum_{r=0}^{k} \alpha_r \right) G^1_0 \cdot G^2_{k-1} = \beta^\alpha_k G^1_0 \cdot G^2_{k-1}
\end{equation}
and
\begin{equation}
Z_k := \left[ G^1_k \cdot \left(  \di \sum_{r=2}^{N-k} \alpha_r G^2_{k+r-1} \right) -\left( \di \sum_{r=2}^{k} \alpha_r G^1_{k-r} \right) \cdot G^2_{k-1} \right].
\end{equation}
Our aim is then to write $\boldsymbol{Y}$ and $\boldsymbol{Z}$ as explicit discrete derivatives. We then define for any $i=0,\ldots,N$, $V_i := h \sum_{r=1}^i Y_r$ and $W_i := h \sum_{j=1}^i Z_j$. Hence, we have $\Delta^1_- \boldsymbol{V} = \boldsymbol{Y}$ and $\Delta^1_- \boldsymbol{W} = \boldsymbol{Z}$ and then, $\boldsymbol{X} = \Delta^1_- (\alpha_1 h \G^1 \cdot \G^2 + \boldsymbol{V} + \boldsymbol{W})$. Our aim is then to explicit $\boldsymbol{V}$ and $\boldsymbol{W}$. For any $i=0,\ldots,N$, we have:
\begin{equation}
V_i = h \di \sum_{r=1}^i \beta^\alpha_r G^1_0 \cdot G^2_{r-1} = h \di \sum_{r=1}^i \beta^\alpha_r G^1_0 \cdot \sigma^{r-1} (\G^2)_{0} = h \di \sum_{r=1}^N \di \sum_{j=0}^N \beta^\alpha_r C_r (i,j) G^1_j \cdot \sigma^{r-1} (\G^2)_{j}.
\end{equation}
For any $i=0,\ldots,N$, we have:
\begin{eqnarray}
W_i & = & h \di \sum_{j=1}^i \left[ G^1_j \cdot \left( \di \sum_{r=2}^{N-j} \alpha_r G^2_{j+r-1} \right) -  \left( \di \sum_{r=2}^{j} \alpha_r G^1_{j-r} \right) \cdot G^2_{j-1}  \right] \\
& = & h \di \sum_{j=1}^i \sum_{r=2}^{N-j} \alpha_r G^1_j \cdot G^2_{j+r-1} - h \sum_{j=2}^i \sum_{r=2}^{j} \alpha_r G^1_{j-r} \cdot G^2_{j-1} \\
& = & h \di \sum_{j=1}^i \sum_{r=2}^{N-j} \alpha_r G^1_j \cdot \sigma^{r-1}(\G^2)_j - h \sum_{r=2}^i \sum_{j=r}^{i} \alpha_r G^1_{j-r} \cdot G^2_{j-1} \\
& = & h \di \sum_{j=1}^i \sum_{r=2}^{N-j} \alpha_r G^1_j \cdot \sigma^{r-1}(\G^2)_j - h \sum_{r=2}^i \sum_{j=0}^{i-r} \alpha_r G^1_{j} \cdot \sigma^{r-1}(\G^2)_j.
\end{eqnarray}
The following equality holds for any $r=2,\ldots,N$ and any $i$, $j=0,\ldots,N$:
\begin{multline}
\delta_{\{ 1 \leq j \leq i \}} \delta_{\{ 2 \leq r \leq N-j \}} - \delta_{\{ 2 \leq r \leq i \}} \delta_{\{ 0 \leq j \leq i-r \}} \\ = \delta_{\{1 \leq i \leq N-1\}} \delta_{\{1 \leq j \leq N-r\}} \delta_{\{0 \leq i-j \leq r-1\}} - \delta_{\{j=0\}} \delta_{\{r \leq i\}}.
\end{multline}
Consequently, we have for any $i=0,\ldots,N$:
\begin{equation}
W_i = h \sum_{r=2}^{N} \sum_{j=0}^N \alpha_r B_r (i,j)  G^1_{j} \cdot \sigma^{r-1}(\G^2)_j.
\end{equation}
Finally, we have for any $i=0,\ldots,N$:
\begin{equation}
\alpha_1 h G^1_i \cdot G^2_i + V_i + W_i = h \sum_{r=1}^N \sum_{j=0}^N A_r (i,j) G^1_j \cdot \sigma^{r-1} (\G^2 )_j.
\end{equation}
Finally, the following equality holds:
\begin{equation}
\boldsymbol{X} = h \Delta^1_- \Big[ \di \sum_{r=1}^N A_r \times \big( \G^1 \cdot \sigma^{r-1} (\G^2 ) \big) \Big],
\end{equation}
which concludes the proof. \\

Now, let us see some examples of matrices $A_r \in \mathcal{M}_{N+1}$ for $N=5$:
\begin{equation*}
A_1 = \left(
\begin{array}{cccccc}
\alpha_1 & 0 & 0 & 0 & 0 & 0 \\
\beta^\alpha_1 & \alpha_1& 0 & 0 & 0 & 0 \\
\beta^\alpha_1 & 0& \alpha_1& 0 & 0 & 0 \\
\beta^\alpha_1 & 0 & 0& \alpha_1& 0 & 0 \\
\beta^\alpha_1 & 0 & 0 & 0& \alpha_1 & 0 \\
\beta^\alpha_1 & 0 & 0 & 0& 0 & \alpha_1 
\end{array} \right) , 
\; A_2 = \left(
\begin{array}{cccccc}
0 & 0 & 0 & 0 & 0 & 0 \\
0 & \alpha_2& 0 & 0 & 0 & 0 \\
\beta^\alpha_2 - \alpha_2 & \alpha_2& \alpha_2& 0 & 0 & 0 \\
\beta^\alpha_2 - \alpha_2& 0 & \alpha_2& \alpha_2& 0 & 0 \\
\beta^\alpha_2 - \alpha_2& 0 & 0 & \alpha_2& 0 & 0 \\
\beta^\alpha_2 - \alpha_2& 0 & 0 & 0& 0 & 0 
\end{array} \right) , 
\end{equation*}
\begin{equation*}
A_3 = \left(
\begin{array}{cccccc}
0 & 0 & 0 & 0 & 0 & 0 \\
0 & \alpha_3& 0 & 0 & 0 & 0 \\
0 & \alpha_3& \alpha_3& 0 & 0 & 0 \\
\beta^\alpha_3 - \alpha_3 & \alpha_3& \alpha_3& 0 & 0 & 0 \\
\beta^\alpha_3 - \alpha_3 & 0 & \alpha_3& 0 & 0 & 0 \\
\beta^\alpha_3 - \alpha_3 & 0 & 0& 0 & 0 & 0 
\end{array} \right) , 
\; A_4 = \left(
\begin{array}{cccccc}
0 & 0 & 0 & 0 & 0 & 0 \\
0 & \alpha_4& 0 & 0 & 0 & 0 \\
0 & \alpha_4& 0 & 0 & 0 & 0 \\
0 & \alpha_4& 0 & 0 & 0 & 0 \\
\beta^\alpha_4 - \alpha_4 & \alpha_4& 0 & 0 & 0 & 0 \\
\beta^\alpha_4 - \alpha_4 & 0 & 0 & 0 & 0 & 0 
\end{array} \right) 
\end{equation*}
and
\begin{equation*}
A_5 = \left(
\begin{array}{cccccc}
0 & 0 & 0 & 0 & 0 & 0 \\
0 & 0& 0 & 0 & 0 & 0 \\
0 & 0& 0& 0 & 0 & 0 \\
0 & 0& 0& 0 & 0 & 0 \\
0 & 0 & 0& 0 & 0 & 0 \\
\beta^\alpha_5 - \alpha_5 & 0 & 0& 0 & 0 & 0 
\end{array} \right) . 
\end{equation*}

\bibliographystyle{plain}

\end{document}